\documentclass[11pt]{amsart}
\usepackage{a4wide,wasysym,amssymb,indentfirst,graphicx,cite,amsthm,color,amsmath}

\usepackage[bookmarksnumbered, linktocpage, plainpages]{hyperref}
\usepackage{epsfig}
\usepackage{float}
\usepackage{mathdots}
\usepackage{extarrows}
\usepackage{indentfirst}

\newtheorem{theorem}{Theorem}[section]

\newtheorem{lemma}[theorem]{Lemma}

\theoremstyle{definition}
\newtheorem{definition}{Definition}[section]

\theoremstyle{remark}
\newtheorem{remark}{Remark}[section]

\newcommand{\RE}{\text{Re}}
\newcommand{\IM}{\text{Im}}
\newcommand{\loc}{\text{loc}}
\newcommand{\osc}{\text{osc}}
\newcommand{\dist}{\text{dist}}
\newcommand{\mcl}{{\mathcal{L}}}
\newcommand{\mbr}{{\mathbb{R}}}
\newcommand{\mbd}{{\mathbb{D}}}
\newcommand{\mbs}{{\mathbb{S}}}

\numberwithin{equation}{subsection} 
\makeatletter
\@addtoreset{equation}{section}
\makeatother

\begin{document}
\title[]
{optimal extensions of conformal mappings from the unit disk to cardioid-type domains}

\author{Haiqing Xu}


\begin{abstract}
The conformal mapping $f(z)=(z+1)^2 $ from $\mathbb{D}$ onto the standard cardioid has a homeomorphic extension of finite distortion to entire $\mathbb{R}^2 .$ We study the optimal regularity of such extensions, in terms of the integrability degree of the distortion and of the derivatives, and these for the inverse. We generalize all outcomes to the case of conformal mappings from $\mathbb{D}$ onto cardioid-type domains. 
\end{abstract}

\maketitle
\section{Introduction}
The standard cardioid domain 
\begin{equation}\label{Delta}
\Delta = \{(x,y) \in \mathbb{R}^2 : (x^2 +y^2)^2 -4x (x^2 +y^2) -4y^2 <0\}
\end{equation}
is the image of the unit disk $\mathbb{D}$ under the conformal mapping $g(z)=(z+1)^2 .$  
Since the origin is an inner-cusp point of $\partial \Delta ,$ the Ahlfors' three-point property fails, and hence $\partial \Delta$ is not a quasicircle. Therefore the preceding conformal mapping does not possess a quasiconformal extension to the entire plane. 
However, there is a homeomorphic extension $f : \mbr^2 \rightarrow \mbr^2 $ by the Schoenflies theorem, see \cite[Theorem 10.4]{Moise 1977}. 
Recall that homeomorphisms of finite distortion 
form a much larger class of homeomorphisms than quasiconformal mappings.   
A natural question arises: can we extend $g$ as a homeomorphism of finite distortion?
If we can, how good an extension can we find?
Our first result gives a rather complete answer.

\begin{theorem}\label{cardioid K_f}
Let $\mathcal{F}$ be the collection of homeomorphisms $f : \mathbb{R}^2 \rightarrow \mathbb{R}^2$ of finite distortion such that $f(z)=(z+1)^2$ for all $z \in \mathbb{D} .$ Then $\mathcal{F} \neq \emptyset .$
Moreover 
\begin{equation}\label{cardioid K_f :0}
\sup \{p \in [1, +\infty): f \in \mathcal{F} \cap W^{1,p} _{\loc} (\mbr^2 ,\mbr^2)\} =+\infty,
\end{equation}
\begin{equation}\label{cardioid K_f :1}
\sup \{q \in (0,+\infty): f \in \mathcal{F},\ K_f \in L^q _{\loc} (\mathbb{R}^2)\} = 2 ,
\end{equation} 
\begin{align}\label{cardioid K_f :2}
& \sup \{q \in (0,+\infty): f \in \mathcal{F}  \cap W^{1,p} _{\loc} (\mathbb{R}^2 , \mbr^2)\mbox{ for some }p>1 \mbox{ and } K_f \in L^q _{\loc} (\mathbb{R}^2)\} \notag \\
& = 1 ,
\end{align}
\begin{equation}\label{cardioid K_f :3} 
\sup \{p \in [1,+\infty):  f \in \mathcal{F},\ f^{-1} \in W^{1,p} _{\loc}(\mathbb{R}^2 , \mbr^2 ) \} = \frac{5}{2}
\end{equation}
and 
\begin{equation}\label{cardioid K_f :4} 
\sup \{q \in (0,+\infty): f \in \mathcal{F},\ K_{f^{-1}} \in L^q _{\loc} (\mathbb{R}^2)\} =  5 .
\end{equation}
\end{theorem}

The cardioid curve $\partial \Delta$ contains an inner-cusp point of asymptotic polynomial degree $3/2 .$ 
Motivated by this, we introduce a family of cardioid-type domains $\Delta_s $ with degree $s >1 ,$ see \eqref{definition of M_s and Delta_s}.
Our second result is an analog of Theorem \ref{cardioid K_f}. 

\begin{theorem}\label{theorem K_f}
Let $g$ be a conformal map from $\mathbb{D}$ onto $\Delta_s ,$ where $\Delta_s $ is defined in \eqref{definition of M_s and Delta_s} and $s >1 .$ 
Suppose that $\mathcal{F}_s (g)$ is the collection of homeomorphisms $f : \mathbb{R}^2 \rightarrow \mathbb{R}^2$ of finite distortion such that $f|_{\mathbb{D}} =g .$ Then $\mathcal{F}_s (g)\neq \emptyset .$  
Moreover 
\begin{equation}\label{theorem |df|}
\sup \{p \in [1, +\infty): f \in \mathcal{F}_s (g) \cap W^{1,p} _{\loc} (\mbr^2 ,\mbr^2)\} =+\infty,
\end{equation}
\begin{equation}\label{theorem K_f: 1}
\sup \{q \in (0,+\infty): f \in \mathcal{F}_s (g),\ K_f \in L^q _{\loc} (\mathbb{R}^2)\} = \max \left\{\frac{1}{s-1},1 \right\},
\end{equation}
\begin{align}\label{theorem Df+K_f :1}
& \sup \{q \in (0, +\infty): f \in \mathcal{F}_s (g) \cap W^{1,p} _{\loc} (\mathbb{R}^2, \mbr^2) \mbox{ for some } p>1 \mbox{ and } K_f \in L^q _{\loc} (\mathbb{R}^2)\} \notag \\
 = & \max \left\{\frac{1}{s-1}, \frac{3p}{(2s-1)p+4-2s} \right\},
\end{align}
\begin{equation}\label{theorem f^-1 : 1} 
\sup \{p \in [1,+\infty):  f \in \mathcal{F}_s (g),\ f^{-1} \in W^{1,p} _{\loc}(\mathbb{R}^2 , \mbr^2) \} = \frac{2(s+1)}{2s-1}
\end{equation}
and 
\begin{equation}\label{theorem f^-1 : 2} 
\sup \{q \in (0,+\infty): f \in \mathcal{F}_s (g) ,\  K_{f^{-1}} \in L^q _{\loc} (\mathbb{R}^2)\} =  \frac{s+1}{s-1} .
\end{equation}
\end{theorem}

Extendability questions similar to Theorem \ref{theorem K_f}
have also been studied in \cite{Guo 2014 Publ. Mat., Guo 2015 Proc. Amer. Math. Soc., Koskela 2007 Publ. Mat.}.

In Section $2$, we recall some basic definitions and facts. We also introduce auxiliary mappings and domains. In Section $3$, we 
give upper bounds for integrability degrees of potential extensions.
Section $4$ is devoted to the proof of Theorem \ref{theorem K_f}. In Section $5,$ we prove Theorem \ref{cardioid K_f}.

\section{Preliminaries}

\subsection{Notation}
By $s \gg  1$ and $t \ll  1$ we mean that $s$ is sufficiently large and $t$ is sufficiently small, respectively.
By $f \lesssim  g$ we mean that there exists a constant $M > 0$ such that $f(x) \le  Mg(x)$ for every $x$. We write $f \approx g$ if both $f \lesssim  g$ and $g \lesssim  f $ hold. 
By $\mathcal{L}^2$ (respectively $\mathcal{L}^1$) we mean the $2$-dimensional ($1$-dimensional) Lebesgue measure.
Furthermore we refer to the disk with center $P$ and radius $r $ by $B(P,r)  ,$ and $S(P,r) = \partial B(P,r) .$ For a set $E \subset \mathbb{R}^2$ we denote by $\overline{E}$ the closure of $E .$ If $A \in \mathbb{R}^{2 \times 2}$ is a matrix, $adj A$ is the adjoint matrix of $A .$

\subsection{Basic definitions and facts}\label{definitions}

\begin{definition}
Let $\Omega \subset \mathbb{R}^2$ and $\Omega' \subset \mathbb{R}^2$ be domains. A homeomorphism $f : \Omega \rightarrow \Omega'$ is called $K$-quasiconformal if $f \in W^{1,2} _{\loc} (\Omega, \mbr^2)$ and if there is a constant $K \ge 1 $ such that  
\begin{equation*}
|Df(z)|^2 \le K J_f (z)
\end{equation*}
holds for $\mathcal{L}^2$-a.e. $z \in \Omega .$
\end{definition}

\begin{definition}
Let $\Omega \subset \mathbb{R}^2$ be a domain.
We say that a mapping $f: \Omega \rightarrow \mathbb{R}^2$ has finite distortion if $f \in W^{1,1} _{\loc} (\Omega , \mathbb{R}^2),$ $J_f \in L^1 _{\loc} (\Omega)$ and 
\begin{equation}\label{MFD :1}
|Df (z)|^2 \le K_f(z) J_f (z) \qquad \mcl^2 \mbox{-a.e. } z \in \Omega,
\end{equation}
where 
\begin{equation*}
K_f (z) = 
\begin{cases}
\frac{|Df(z)|^2}{J_f (z)} & \mbox{for all } z \in \{J_f >0 \}, \\
1& \mbox{for all } z \in \{J_f =0 \}.
\end{cases}
\end{equation*}
\end{definition}

\begin{definition}
Given $A \subset \mathbb{R}^2 ,$ a map $f : A \rightarrow \mathbb{R}^2$ is called an $(l,L)$-bi-Lipschitz mapping if
$0 <l \le L <\infty$ and 
\begin{equation*}
l |x-y| \le |f(x) -f(y)| \le L|x-y|
\end{equation*}
for all $x, y \in A .$ 
\end{definition}
If $\Omega \subset \mbr^2$ is a domain and $f: \Omega \rightarrow \mbr^2$ is an orientation-preserving bi-Lipschiz mapping, then $f$ is quasiconformal. 

\begin{definition}
Given a function $\varphi$ defined on set $A \subset \mathbb{R}^2,$ its modulus of continuity is defined as
\begin{equation*}
\omega(\delta) \equiv \omega(\delta, \varphi, A) = \sup\{|\varphi(z_1) -\varphi(z_2)|: z_1 ,z_2 \in A,\ |z_1 -z_2| \le \delta\}
\end{equation*}
for $\delta \ge 0.$
Then $\varphi$ is called Dini-continuous if 
\begin{equation*}
\int_{0} ^{\pi} \frac{\omega(t)}{t}\, dt < \infty,
\end{equation*}
where the integration bound $\pi$ can be replaced by any positive constant.

We say that a curve $C$ is $\mathit{Dini}$-$\mathit{smooth}$ if it has a parametrization $\alpha(t)$ for $t \in [0,2\pi]$ so that $\alpha'(t) \neq 0$ for all $ t \in [0,2\pi]$ and $\alpha'$ is Dini-continuous. 
\end{definition}

\begin{definition}
Let $\Omega \subset \mathbb{R}^2$ be open and $f: \Omega \rightarrow \mathbb{R}^2$ be a mapping. We say that $f$ satisfies the Lusin ($N$) condition if $\mathcal{L}^2 (f(E))=0$ for any $E \subset \Omega$ with $\mathcal{L}^2 (E)=0 .$
Similarly, $f$ satisfies the Lusin ($N^{-1}$) condition if $\mathcal{L}^2 (f^{-1}(E))=0$ for any $E \subset \Omega$ with $\mathcal{L}^2 (E)=0 .$ 
\end{definition}

\begin{lemma}{\rm{(\cite[Theorem A.35]{Hencl 2014})}}\label{change of variables}\label{lemma A}
Let $\Omega \subset \mathbb{R}^2$ be open and $f \in W^{1,1} _{\loc} (\Omega, \mathbb{R}^2).$ Suppose that $\eta$ is a nonnegative Borel measurable function on $\mathbb{R}^2 .$ Then 
\begin{equation}\label{area formula: 1}
\int_{\Omega} \eta(f(x)) |J_f (x)|\, dx \le \int_{f(\Omega)} \eta(y) N(f,\Omega,y)\, dy,
\end{equation}
where the multiplicity function $N(f,\Omega,y)$ of $f$ is defined as the number of preimages of $y$ under $f$ in $\Omega .$ Moreover \eqref{area formula: 1} is an equality if we assume in addition that $f$ satisfies the Lusin ($N$) condition.
\end{lemma}

\begin{lemma}{\rm{(\cite[Lemma A.28]{Hencl 2014})}}\label{diff. a.e.}\label{lemma B}
Suppose that $f: \mathbb{R}^2 \rightarrow \mathbb{R}^2$ is a homeomorphism which belongs to $W^{1,1} _{\loc} (\mathbb{R}^2 , \mathbb{R}^2).$ Then $f$ is differentiable $\mathcal{L}^2$-a.e. on $\mathbb{R}^2$.
\end{lemma}
Lemma \ref{lemma B} and a simple computation show that 
\begin{equation}\label{MFD :3}
\max_{\theta \in [0,2\pi]}|\partial _{\theta} f (z)| =K_f (z) \min_{\theta \in [0,2\pi]}|\partial _{\theta} f (z)| \qquad \mcl^2 \mbox{-a.e. } z \in \mbr^2
\end{equation}
when $f: \mbr^2 \rightarrow \mbr^2$ is a homeomorphism of finite distortion. Here $\partial_{\theta} f (z)=  \cos( \theta) f_x (z)  +  \sin (\theta) f_y (z)$ for $\theta \in [0, 2\pi] .$

\begin{lemma}{\rm{(\cite[Theorem 1.2]{Hencl 2006 Arch. Ration. Mech. Anal.}, \cite[Theorem 1.6]{Hencl 2014})}} \label{lemma C}
Let $\Omega \subset \mathbb{R}^2$ be a domain and $f: \Omega \rightarrow \mbr^2$ be a homeomorphism of finite distortion. 
Then $f^{-1}: f(\Omega) \rightarrow \Omega$ is also a homeomorphism of finite distortion. Moreover 
\begin{equation}\label{dist. ine for inv}
|Df^{-1} (y)|^2 \le K_{f^{-1}} (y) J_{f^{-1}} (y) \qquad \mcl^2 \mbox{-a.e. } y \in f(\Omega).
\end{equation}
\end{lemma}

\begin{lemma}{\rm{(\cite[Theorem 2.1.11]{Ziemer 1989})}}\label{lemma D}  
Let all $\Omega \subset \mathbb{R}^2, \ \Omega_1 \subset \mathbb{R}^2$ and $\Omega_2 \subset \mathbb{R}^2$ be open, and $T \in Lip (\Omega_1 , \Omega_2).$ Suppose that both $f \in W^{1,p} _{\loc} (\Omega , \Omega_1)$ and $T \circ f \in L^{p} _{\loc} (\Omega ,\Omega_2)$ hold for some $p$ with $1 \le p \le \infty.$ Then $T \circ f \in W^{1,p} _{\loc} (\Omega, \Omega_2 )$ and 
\begin{equation*}
D(T \circ f)(z) = DT(f(z)) Df(z) \qquad \mathcal{L}^2\mbox{-a.e. } z \in \Omega.
\end{equation*}
\end{lemma}

\begin{definition}
A rectifiable Jordan curve $\Gamma$ in the plane is a chord-arc curve if there is a constant $C>0$ such that
\begin{equation*}
\ell_{\Gamma} (z_1 ,z_2) \le C |z_1 -z_2|
\end{equation*}
for all $z_1 , z_2 \in \Gamma ,$ where $\ell_{\Gamma} (z_1 ,z_2)$ is the length of the shorter arc of $\Gamma$ joining $z_1$ and $z_2 .$
\end{definition}
It is a well-known fact that a chord-arc curve is the image of the unit circle under a bi-Lipschitz mappings of the plane, see \cite{Jerison 1982 Math. Scand.}. Thus chord-arc curves form a special class of quasicircles. The connections between chord-arc curves and quasiconformal theory can be found in \cite{Astala 2016 Invent. Math., Semmes 1988 Trans. Amer. Math. Soc.}.

\subsection{Definition of cardioid-type domains} 
Let $s >1 .$ 
We introduce a class of cardioid-type domains $\Delta_s$ whose boundaries contain internal polynomial cusps of order $s$, see FIGURE \ref{m_s and delta_s}. For technical reasons we do this in the following manner.
Denote
\begin{equation*}
\ell_1 (s) = \{(u,v)\in \mathbb{R}^2 : u \in [-1, 0] ,\ v=(-u)^s\} 
\end{equation*}
and
\begin{equation*}
\ell_2 (s) = \{(u,v) \in \mathbb{R}^2 : u \in [-1, 0] ,\ v=-(-u)^s\}.
\end{equation*} 
Write $\ell_1 (s)$ and $\ell_2 (s)$ in the polar coordinate system as
\begin{align*}
\ell_1(s) = \{  R e^{i \Theta}:\ & R= (-u) (1+(-u)^{2(s-1)})^{\frac{1}{2}}  \\
& \mbox{ and }\Theta =\pi- \arctan ((-u)^{s-1}) \mbox{ for } u \in [-1, 0]\}
\end{align*}
and 
\begin{align*}
\ell_2 (s)= \{  R e^{i \Theta}:\ & R= (-u) (1+(-u)^{2(s-1)})^{\frac{1}{2}}  \\
& \mbox{ and }\Theta =-\pi+ \arctan ((-u)^{s-1}) \mbox{ for } u \in [-1, 0]\}. 
\end{align*}
Take the branch of complex-valued function $z = w^{1/2}$ with $1^{1/2} =1.$ Denote by $\ell^m _1 (s)$ and $\ell^m _2 (s)$ the images of $\ell_1 (s)$ and $\ell_2 (s)$ under the preceding $z=w^{1/2},$ respectively. Then we can write $\ell^m _1 (s)$ and $\ell^m _2 (s)$ in the polar coordinate system as
\begin{align}\label{intro: ell^m _1}
\ell^m _1 (s)=\{re^{i \theta}:\ & r= \sqrt{-u} (1+(-u)^{2(s-1)})^{\frac{1}{4}} \notag \\ & \mbox{ and } \theta = \frac{\pi- \arctan ((-u)^{s-1})}{2} \mbox{ for } u \in [-1, 0]\}
\end{align} 
and 
\begin{align*}
\ell^m _2 (s)=\{re^{i \theta}:\ & r= \sqrt{-u} (1+(-u)^{2(s-1)})^{\frac{1}{4}} \\ & \mbox{ and } \theta = \frac{-\pi+ \arctan ((-u)^{s-1})}{2} \mbox{ for } u \in [-1, 0]\}.
\end{align*} 
Denote by $z_1$ and $z_2$ the end points of $\ell^m _1 (s)\cup \ell^m _2 (s).$
Notice that there is a unique circle sharing both the tangent of $\ell^m _1 (s)$ at $z_1$ and the one of $\ell^m _2 (s)$ at $z_2.$ This circle is divided into two arcs by $z_1$ and $z_2.$ 
Concatenating $\ell^m _1 (s) \cup \ell^m _2 (s)$ with the arc located on the right-hand side of the line through $z_1$ and $z_2$, we then obtain a Jordan curve $\ell^m (s).$ 
Denote by $\ell (s)$ the image of $\ell^m (s)$ under $z^2.$ Let
\begin{equation}\label{definition of M_s and Delta_s}
M_s \mbox{ and } \Delta_s \mbox{ be the interior domains of } \ell^m (s) \mbox{ and }\ell (s),\ \mbox{respectively.}
\end{equation}
Then $\Delta_s$ is the desired cardioid-type domain with degree $s$. Moreover $\ell^m (s),\ \ell(s) ,\ M_s$ and $\Delta_s$ are symmetric with respect to the real axis.

\begin{figure}[htbp]
\centering \includegraphics[bb=136 547 583 724, clip=true, scale=0.8]{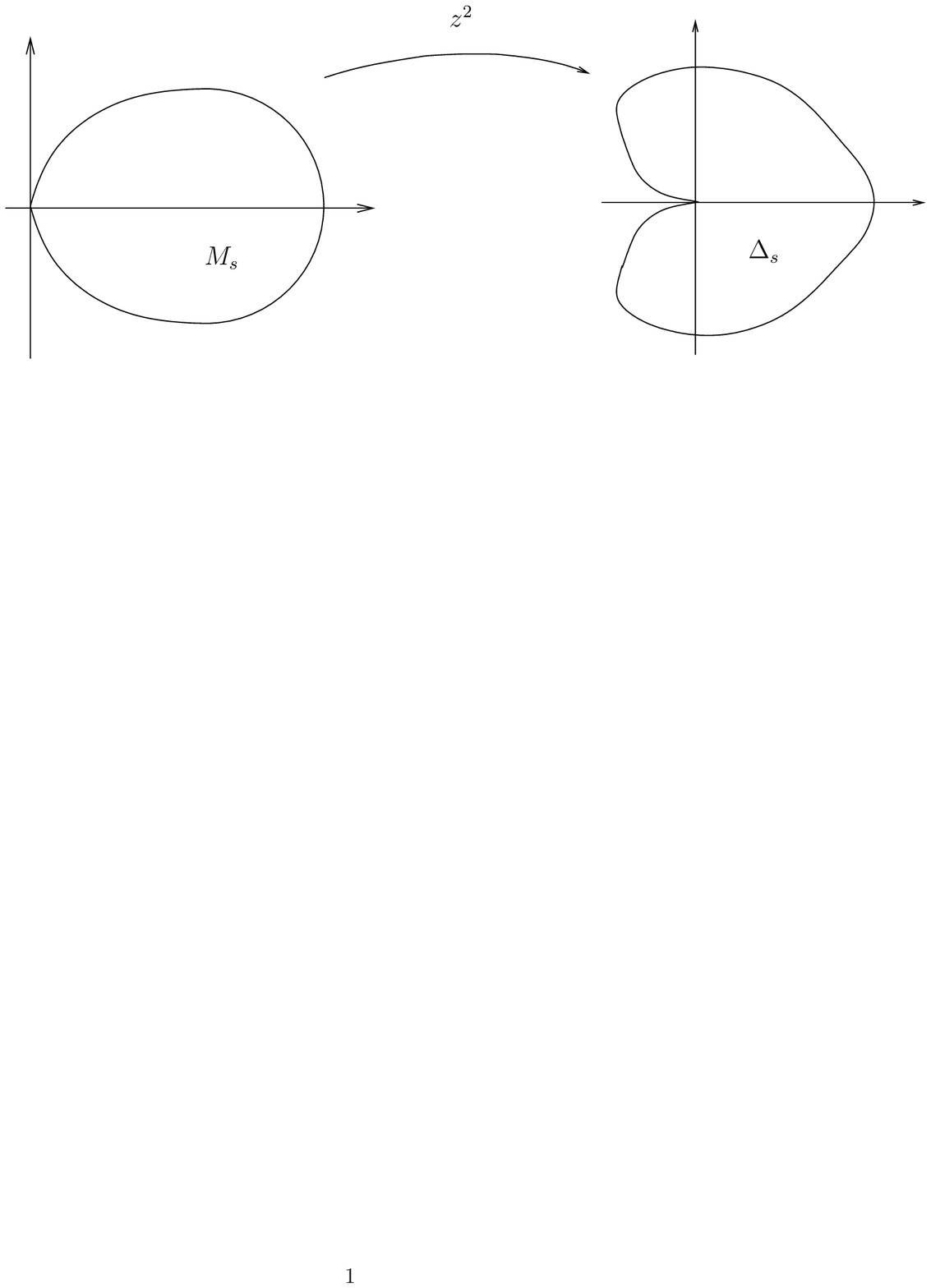}

\caption{$M_s$ and $\Delta_s$}
\label{m_s and delta_s}
\end{figure}

By the Riemann mapping theorem, there is a conformal mapping from $\mathbb{D} \cap \mathbb{R}^2 _+$ onto $M_s \cap \mathbb{R}^2 _+$ such that $\mathbb{D} \cap \mathbb{R}$ is mapped onto $M_s \cap \mathbb{R}.$ 
It follows from the Schwarz reflection principle that there is a conformal mapping 
\begin{equation}\label{g_s}
g_s : \mathbb{D} \rightarrow M_s.  
\end{equation}
such that $g_s (\bar{z}) = \overline{g_s (z)}$ for all $z \in \mbd .$
Moreover by the Osgood-Carath\'eodory theorem $g_s$ has a homeomorphic extension from $\overline{\mathbb{D}}$ onto $\overline{M_s},$ still denoted $g_s .$ 

\begin{lemma}\label{lemma: g_s and g^c _s}
Let $M_s$ and $g_s$ be as in \eqref{definition of M_s and Delta_s} and \eqref{g_s} with $s >1 .$ Then $g_s $ is a bi-Lipschitz mapping on $\overline{\mathbb{D}} .$
\end{lemma}

\begin{proof}
If $\partial M_s$ were a Dini-smooth Jordan curve, 
from \cite[Theorem 3.3.5]{Pommerenke 1992} it would follow that $g' _s$ is continuous on $\overline{\mathbb{D}}$ and $g' _s (z) \neq 0$ for all $z \in \overline{\mathbb{D}}.$ 
Since $M_s$ is convex, the mean value theorem would then yield that  
$g_s$ is a bi-Lipschitz map from $\overline{\mathbb{D}}$ onto $\overline{M_s}.$

In order to prove that $\partial M_s$ is a Dini-smooth Jordan curve,
we first analyze $\partial M_s$ in a neighborhood of the origin.  
For any point in $\ell^m _1$ with Euclidean coordinate $(x,y),$ 
we have 
\begin{equation}\label{intro: 5}
x= r \cos \theta \mbox{ and } y= r \sin \theta.
\end{equation}
where both $r$ and $\theta$ share the expression in \eqref{intro: ell^m _1}.
We then obtain that
\begin{equation}\label{intro: 6}
r \approx \sqrt{-u},\ \theta \approx \frac{\pi}{2},\ 
\frac{\partial r}{\partial u} \approx \frac{-1}{\sqrt{-u}} \mbox{ and } \frac{\partial \theta}{\partial u} \approx (-u)^{s-2}
\end{equation}
whenever $|u| \ll 1.$
Therefore from \eqref{intro: 5} and \eqref{intro: 6}, it follows that 
\begin{equation*}
x \approx (-u)^{s- \frac{1}{2}},\ 
y \approx (-u)^{\frac{1}{2}},\ 
\frac{\partial x}{\partial u} \approx - (-u)^{s- \frac{3}{2}} \mbox{ and }
\frac{\partial y}{\partial u} \approx -(-u)^{-\frac{1}{2}}.
\end{equation*}
Together with symmetry of $\partial M_s ,$
we conclude that $\frac{\partial x}{\partial y} \approx |y|^{2(s-1)}$ whenever $|y| \ll 1.$
Next, notice that the part of $\partial M_s$ away from the origin is piecewise smooth.
By parametrizing $\partial M_s$ as $\alpha(y)=(x(y),y),$ we then obtain that the modulus of continuity of $\alpha'$ satisfies 
\begin{equation*}
\omega(\delta,\alpha',\partial M_s) \le \max\{\delta^{2(s-1)}, \delta\} \qquad \forall \delta \ll 1.
\end{equation*}
Consequently $\alpha'$ is Dini-continuous. 
Therefore $\partial M_s$ is a Dini-smooth Jordan curve.
\end{proof}

\begin{remark}\label{deduce to M_s onto Delta_s}
Since $g_s : \mathbb{S}^1 \rightarrow \partial M_s$ is a bi-Lipschitz map by Lemma \ref{lemma: g_s and g^c _s}, via \cite[Theorem A]{Tukia 1980 Ann. Acad. Sci. Fenn. Ser. A I Math.} there is a bi-Lipschitz mapping $g^c _s : \mathbb{D}^c \rightarrow M_s ^c$ such that $g^c _s |_{\mathbb{S}^1} =g_s .$
Let
\begin{equation}\label{rmk inverse: 0}
G_s (z)=
\begin{cases}
g_s (z)  &\forall z \in \overline{\mathbb{D}}, \\
g^c _s (z)  &\forall z \in \mathbb{D}^c .
\end{cases}
\end{equation}
Then $G_s$ is an orientation-preserving bi-Lipschitz mapping. 
\end{remark}

\begin{lemma}\label{composition}
Let $h_1 : \mathbb{R}^2 \rightarrow \mathbb{R}^2$ be a homeomorphism of finite distortion, and $h_2 : \mathbb{R}^2 \rightarrow \mathbb{R}^2$ be an $(l,L)$-bi-Lipschitz, orientation-preserving mapping. Then $h_1 \circ h_2$ is a homeomorphism of finite distortion. 
\end{lemma}

\begin{proof}
Since $h_2$ is an orientation-preserving bi-Lipschitz mapping, we have that $h_2$ is quasiconformal. 
From \cite[Corollary 3.7.6]{Astala 2009} it then follows that 
\begin{equation}\label{composition: 2}
h_2 \mbox{ satisfies Lusin }(N) \mbox{ and }(N^{-1})\mbox{ condition,}
\end{equation}
\begin{equation}\label{composition: 2-1}
J_{h_2} >0 \quad \mathcal{L}^2 \mbox{-a.e. on }\mathbb{R}^2.  
\end{equation}
By Lemma \ref{lemma B} we have 
\begin{equation}\label{composition: 1}
\mbox{both } h_1 \mbox{ and } h_2 \mbox{ are differentiable }\mathcal{L}^2 \mbox{-a.e. on }\mathbb{R}^2 .
\end{equation}
From \eqref{composition: 1} and \eqref{composition: 2} it therefore follows that $h_1  \circ h_2$ is differentiable $\mathcal{L}^2$-a.e. on $\mathbb{R}^2,$ and
\begin{equation}\label{composition: 3}
D(h_1  \circ h_2) (z) =Dh_1 (h_2 (z)) D h_2 (z) \qquad \mathcal{L}^2 \mbox{-a.e. } z \in \mathbb{R}^2 .
\end{equation}
By \eqref{composition: 3}, Lemma \ref{lemma A} and \eqref{composition: 2}, we then have that
\begin{equation}\label{composition: 4}
\int_{M} |J_{h_1  \circ h_2} (z) |\, dz = \int_{M} |J_{h_1}  (h_2 (z))| |J_{h_2} (z)|\, dz  = \int_{h_2 (M)} |J_{h_1}  (w)|\, dw <\infty
\end{equation}
for any compact set $M \subset \mathbb{R}^2,$ where the last inequality is from $J_{h_1}  \in L^{1} _{\loc} .$
Moreover, from \eqref{composition: 3} and the distortion inequalities for $h_1$ and $h_2 $ it follows that 
\begin{align}\label{composition: 5}
|D(h_1  \circ h_2) (z)|^2  \le & |Dh_1 (h_2 (z))|^2 |Dh_2 (z)|^2 \le K_{h_1}  (h_2 (z)) K_{h_2} (z) J_{h_1} (h_2 (z)) J_{h_2} (z) \notag\\
= & K_{h_1}  (h_2 (z)) K_{h_2} (z) J_{h_1  \circ h_2} (z)
\end{align}
for $\mathcal{L}^2$-a.e. $z \in \mathbb{R}^2 .$

To prove that $h_1  \circ h_2$ is a homeomorphism of finite distortion, via \eqref{composition: 4} and \eqref{composition: 5} it is sufficient to prove that $h_1  \circ h_2 \in W^{1,1} _{\loc} .$
Since $h_2$ is an $(l,L)$-bi-Lipschitz orientation-preserving mapping, by \eqref{composition: 1} and \eqref{MFD :3} we then have that
\begin{equation}\label{composition: 6}
l \le |Dh_2 (z)| \le L \mbox{ and } 1 \le K_{h_2} (z) \le \frac{L}{l} \qquad \mathcal{L}^2 \mbox{-a.e. }z \in \mathbb{R}^2 .
\end{equation}
From\eqref{composition: 2-1}, \eqref{composition: 6} and \eqref{MFD :1} it then follows that
\begin{equation}\label{composition: 7}
\frac{l^3}{L} \le J_{h_2} (z) \le L^2\qquad \mathcal{L}^2 \mbox{-a.e. } z \in \mathbb{R}^2 .
\end{equation}
By \eqref{composition: 3}, \eqref{composition: 6}, \eqref{composition: 7} and Lemma \ref{lemma A},
we therefore have 
\begin{align*}
\int_{M} |D(h_1  \circ h_2) (z)|\, dz \le & \int_{M} |Dh_1 (h_2 (z))| \frac{|Dh_2 (z)| }{J_{h_2} (z)} J_{h_2} (z)\,  dz \\
\approx &  \int_{M} |Dh_1 (h_2 (z))| J_{h_2} (z) \,  dz \\
= & \int_{h_2(M)} |Dh_1 (w)|\, dw  <\infty
\end{align*}
for any compact set $M \subset \mathbb{R}^2 ,$ where the last inequality is from $h_1  \in W^{1,1} _{\loc} .$
\end{proof}

\section{Bounds for integrability degrees}
For a given $s>1 ,$ let $M_s$ as in \eqref{definition of M_s and Delta_s}. Define
\begin{align}\label{mathcalE_s}
\mathcal{E}_s = \{f : \ & f: \mathbb{R}^2 \rightarrow \mathbb{R}^2 \mbox{ is a homeomorphism of finite distortion}   \notag \\
& \mbox{and } f(z)=z^2 \mbox{ for all } z \in \overline{M_s} \}.
\end{align}

\begin{lemma}\label{negative part of f^-1}
Let $\mathcal{E}_s$ be as in \eqref{mathcalE_s} with $s > 1 ,$ and 
$f \in \mathcal{E}_s .$ Suppose that $f^{-1} \in W^{1,p}_{\loc} (\mathbb{R}^2, \mbr^2)$ for some $p \ge 1.$ Then necessarily $p < 2(s+1)/(2s-1).$
\end{lemma}

\begin{proof}
Given $x \in (-1 ,0),$ denote by $I_x$ the line segment connecting the points $(x,|x|^{s})$ and $(x,-|x|^{s}).$ 
Since $f^{-1} \in W^{1,p}_{\loc} $ for some $p \ge 1,$ by the ACL-property of Sobolev functions it follows that
\begin{equation}\label{negative part of f^-1: 1}
\mbox{osc}_{I_x} f^{-1} \le \int_{I_x} |D f^{-1}(x,y)|\, dy
\end{equation}
holds for $\mathcal{L}^1$-a.e. $x \in (-1 ,0).$ 
Applying Jensen's inequality to \eqref{negative part of f^-1: 1}, we have 
\begin{equation}\label{negative part of f^-1: 2}
\frac{(\mbox{osc}_{I_x} f^{-1})^p}{(-x)^{s(p-1)}} \le \int_{I_x} |D f^{-1} (x,y)|^p \, dy.
\end{equation}
Since $f(z) =z^2$ for all $z \in \partial M_s ,$  
we have 
\begin{equation}\label{negative part of f^-1: 2-1}
(-x)^{1/2} \lesssim \mbox{osc}_{I_x} f^{-1} \qquad  \forall x \in (-1 ,0).
\end{equation}
Combining \eqref{negative part of f^-1: 2} with \eqref{negative part of f^-1: 2-1}, we hence obtain 
\begin{equation}\label{negative part of f^-1: 3}
(-x)^{\frac{p}{2} -s(p-1)} \lesssim \int_{I_x} |D f^{-1} (x,y)|^p \, dy \qquad  \mathcal{L}^1 \mbox{-a.e. } x \in (-1 ,0).
\end{equation}
Integrating \eqref{negative part of f^-1: 3} with respect to $x \in (-1 ,0)$ therefore implies 
\begin{equation}\label{negative part of f^-1: 4}
\int_{-1} ^{0}(-x)^{\frac{p}{2} -s(p-1)}\, dx 
\lesssim \int_{B(0,\sqrt{2})} |D f^{-1} (x,y)|^p \, dx \, dy.
\end{equation}
Since $f^{-1} \in W^{1,p}_{\loc} ,$
from \eqref{negative part of f^-1: 4} we necessarily obtain $\frac{p}{2}-s(p-1) >-1,$ which is equivalent to $p < 2(s+1)/(2s-1).$
\end{proof}

Our next proof borrows some ideas from \cite[Theorem 1]{Koskela 2010 Acta Math. Sin.}.
\begin{lemma}\label{necessary K_f^-1}
Let $\mathcal{E}_s$ be as in \eqref{mathcalE_s} with $s > 1 .$ 
Let $f \in \mathcal{E}_s$ and suppose that $K_{f^{-1}} \in L^q _{loc} (\mathbb{R}^2)$ for a given $q \ge 1.$ 
Then $q < (s+1)/(s-1).$
\end{lemma}

\begin{proof}
For a given $t \ll 1,$ we denote 
\begin{equation*}
E_t = \{(x,y) \in \mathbb{R}^2 : x \in (-t^2 , -(\frac{t}{2})^2 ) \mbox{ and } y=-|x|^s\} 
\end{equation*}
and 
\begin{equation*}
F_t = \{(x,y)\in \mathbb{R}^2 : x \in (-t^2 , -(\frac{t}{2})^2 ) \mbox{ and } y=|x|^s\}. 
\end{equation*}
Let $\tilde{E}_t = f^{-1} (E_t) \mbox{ and } \tilde{F}_t = f^{-1} (F_t) .$
Set
\begin{equation*}
 L^1 _t = \min \{|z| : z \in \tilde{F}_t\},\ L^2 _t = \max\{|z|: z \in \tilde{F}_t\},
\end{equation*}
\begin{equation*}
L_t = \dist(\tilde{E}_t,\tilde{F}_t),\  L_0 = \max\{|f^{-1} (z)|: \RE z=-1, \IM z \in [-1,1]  \}.
\end{equation*}
Since $f(z)=z^2$ for all $ z \in \partial M_s ,$ we have  
$L^1 _t \approx t/2,\ L^2 _t \approx t$ and $L_t
 \approx t$ whenever $t \ll 1 .$ 
Given $w \in A_t :=\{w \in \mathbb{R}^2 : L^1 _t \le |w| \le L^2 _t \} ,$ set $\rho (w) = L^2 _t  /(L_t |w|).$ Define
\begin{equation}\label{K_f^-1: v}
v(z) = 
\begin{cases}
1 &\mbox{ for all } z \in B(0,L_0) \setminus A_t, \\
\inf_{\gamma_z} \int_{\gamma_z} \rho \, ds & \mbox{ for all } z \in A_t, 
\end{cases}
\end{equation}
where the infimum is taken over all curves $\gamma_z \subset A_t$ joining $z$ and $\tilde{E}_t.$
From \eqref{K_f^-1: v} it follows that for any $z_1,\ z_2 \in A_t$ and any curve $\gamma_{z_1 z_2} \subset A_t$ connecting $z_1$ and $z_2 $ we have 
\begin{equation}\label{K_f^-1: 1}
|v(z_1) -v(z_2)|\le \int_{\gamma_{z_1 z_2}} \rho \, ds .
 \end{equation}
Therefore $v$ is a Lipschitz function on $A_t .$
By Rademacher's theorem, $v$ is differentiable $\mathcal{L}^2$-a.e. on $A_t.$ 
Hence \eqref{K_f^-1: 1} together with the continuity of $\rho$ gives
\begin{equation}\label{K_f^-1: 2}
|Dv(z)| \le \rho(z) \qquad \mathcal{L}^2 \mbox{-a.e. } z \in A_t . 
\end{equation}
Integrating \eqref{K_f^-1: 2} over $\tilde{Q} _t = A_t \setminus M_s $ then yields
\begin{equation}\label{K_f^-1: 3}
\int_{\tilde{Q} _t} |Dv|^2 \le \int_{\tilde{Q} _t} \rho^2 \approx \int_{L^1 _t} ^{L^2 _t} \frac{1}{r}\, dr \approx \log 2 .
\end{equation}

By Lemma \ref{lemma C} we have $f^{-1} \in W^{1,1} _{\loc} .$ 
Let $u =v \circ f^{-1} .$ 
From Lemma \ref{lemma D} we then have $u \in W^{1,1} _{\loc} (f (B(0,L_0)))$ and 
\begin{equation}\label{K_f^-1: 4-1}
|Du(z)| \le |Dv(f^{-1} (z))| |Df^{-1}(z)| \qquad \mathcal{L}^2 \mbox{-a.e. in }  f (A_t) .
\end{equation} 
By \eqref{K_f^-1: v}, $v(z) =0$ for all $z \in \tilde{E}_t .$
Hence $u (z) =0$ for all $ z \in E_t .$
Whenever $z \in \tilde{F}_t, $ we have $\mcl^1 (\gamma_z ) \ge L_t $ for any curve $\gamma_z \subset A_t$ joining $z$ and $\tilde{E}_t .$ 
Therefore $v(z) \ge 1$ for all $  z \in \tilde{F}_t .$
Hence $u(z)\ge 1$ for all $ z \in F_t .$
By the ACL-property of Sobolev functions and H\"older's inequality, we therefore have that
\begin{equation}\label{K_f^-1: 4}
1 \le \int_{-x^s} ^{x^s} |Du(x,y)| \, dy \le \left(\int_{-x^s} ^{x^s} |Du(x,y)|^p \, dy \right)^{\frac{1}{p}} (2 x^s)^{\frac{p-1}{p}}
\end{equation}
for any $p > 1$ and $\mcl^1$-a.e. $x \in [-t^2 ,-(t/2)^2].$
Define
\begin{equation*}
R_t = \{(x,y)\in \mathbb{R}^2 : x \in (-t^2, -(t/2)^2) ,\ y \in (-|x|^s, |x|^s)\}.
\end{equation*}
Fubini's theorem and \eqref{K_f^-1: 4} then give
\begin{align}\label{K_f^-1: 5}
\int_{R_t} |Du(x,y)|^p \, dx \, dy = & \int_{-t^2} ^{-(t/2)^2}  \int_{-x^s} ^{x^s} |Du(x,y)|^p \, dy \, dx  \notag\\
\gtrsim & \int_{-t^2} ^{-(t/2)^2} x^{s(1-p)} \, dx \approx t^{2(1+s(1-p))}.
\end{align}
Set $Q_t = f(\tilde{Q}_t).$ Then for any $z \in R_t \setminus Q_t$ there is an open disk $B_z \subset R_t \setminus Q_t$ such that $z \in B_z $ and $u|_{B_z} \equiv 1 .$
Therefore 
\begin{equation}\label{K_f^-1: 6}
\int_{Q_t} |Du|^p  \ge  \int_{Q_t \cap R_t} |Du|^p  = \int_{R_t} |Du|^p .
\end{equation}
Combining \eqref{K_f^-1: 5} with \eqref{K_f^-1: 6} gives that
\begin{equation}\label{K_f^-1: 7}
t^{2(1+s(1-p))} \lesssim \int_{Q_t} |Du|^p  
\end{equation}
for all $p \ge 1.$

For any $p \in (0,2),$ by \eqref{K_f^-1: 4-1}, \eqref{dist. ine for inv} and H\"older's inequality we have 
\begin{align}\label{K_f^-1: 8}
\int_{Q_t} |D u|^p \le &\int_{Q_t}|D v \circ f^{-1}|^p |D f^{-1}|^p \notag \\
\le & \int_{Q_t}|D v \circ f^{-1}|^p J^{\frac{p}{2}}_{f^{-1}} K^{\frac{p}{2}}_{f^{-1}} \notag \\
\le & \left(\int_{Q_t} |D v \circ f^{-1}|^2 J_{f^{-1}} \right)^{\frac{p}{2}} \left(\int_{Q_t} K^{\frac{p}{2-p}}_{f^{-1}} \right)^{\frac{2-p}{2}} \notag\\
\le & \left(\int_{\tilde{Q} _t} |D v |^2 \right)^{\frac{p}{2}} \left(\int_{Q_t} K^{\frac{p}{2-p}}_{f^{-1}} \right)^{\frac{2-p}{2}}
\end{align}
where the last inequality comes from Lemma \ref{lemma A}.
Let $q = p/(2-p) .$ Via \eqref{K_f^-1: 3} and \eqref{K_f^-1: 7}, 
we conclude from \eqref{K_f^-1: 8} that
\begin{equation}\label{K_f^-1: 9}
t^{2(1+q +s(1-q))}  \lesssim \int_{Q_t} K^q _{f^{-1}} 
\end{equation}
for all $q \ge 1.$
We now consider the set $Q_t$ for $t =2^{-j}$ with $j \ge j_0$ for a fixed large $j_0 .$
Since 
\begin{equation*}
\sum_{j =j_0} ^{\infty} \chi_{Q_{2^{-j}}} (x) \le 2 \chi_{\mathbb{D}} (x) \qquad \forall x \in \mbr^2 ,
\end{equation*}
by \eqref{K_f^-1: 9} we have that
\begin{equation}\label{K_f^-1: 10}
\sum_{j=j_0} ^{+\infty} 2^{j2(s(q-1)-q-1)} \lesssim 
\sum_{j=j_0} ^{+\infty}  \int_{Q_{2^{-j}}} K^q _{f^{-1}} \le  2\int_{\mathbb{D}} K^q _{f^{-1}}.
\end{equation}
The series in \eqref{K_f^-1: 10} diverges when $q \ge \frac{s+1}{s-1}$ and hence 
$K_{f^{-1}} \in L^q _{\loc} (\mathbb{R}^2)$ can only hold when 
 $q < (s+1)/(s-1) .$
\end{proof}

We continue with properties of our homeomorphism $f .$ The following lemma is a version of \cite[Theorem 4.4]{Guo 2014 Publ. Mat.}.
\begin{lemma}\label{necessary K_f}
Let $\mathcal{E}_s$ be as in \eqref{mathcalE_s} with $s > 1 .$ 
If $f \in \mathcal{E}_s$ and $K_f \in L^q _{loc} (\mathbb{R}^2)$ for some $q \ge 1,$ 
then $q < \max \{1, 1/(s-1)\}.$
\end{lemma}

\begin{proof}
Denote 
\begin{equation*}
\Omega= \{(x_1,x_2)\in \mathbb{R}^2 : x_1 \in (-1, 0),\ x_2  \in (-|x_1|^s, |x_1|^s) \} .
\end{equation*}
For a given $t \ll 1,$ set
\begin{equation*}
\Omega^1 _t = \{(x_1,x_2)\in \Omega : x_1 \in (-1, -t^2) \}, 
\end{equation*}
\begin{equation*}
\tilde{Q} _t= \{(x_1,x_2)\in \Omega : x_1 \in [-t^2, -(\frac{t}{2})^2] \} \mbox{ and }
\Omega^2 _t =\Omega \setminus (\Omega^1 _t \cup \tilde{Q} _t).
\end{equation*}
Define
\begin{equation}\label{negative K_f :-1}
v(x_1,x_2)=
\begin{cases}
1 & \forall (x_1,x_2) \in \Omega^1 _t ,\\
1- \left( \int_{-t^2} ^{- (t/2)^2 } \frac{dx}{(-x)^{s}} \right)^{-1} \int_{-t^2} ^{x_1} \frac{dx}{(-x)^{s}} & \forall (x_1,x_2) \in \tilde{Q} _t ,\\
0 & \forall (x_1, x_2) \in \Omega^2 _t.
\end{cases}
\end{equation}
Then $v$ is a Lipschitz function on $\Omega .$
Let $u=v \circ f.$ By Lemma \ref{lemma D}, we have $u \in W^{1,1} _{\loc} (f^{-1} (\Omega))$ and 
\begin{equation}\label{negative K_f :0-2}
Du(z) = Dv (f(z)) Df(z) \qquad \mathcal{L}^2 \mbox{-a.e. }z \in f^{-1} (\Omega).
\end{equation}
Let $P_1 = f^{-1} ((-t^2,t^{2s})),\ P_2 = f^{-1} ((- (t /2)^2 , (t/2)^{2s}))$ and $O$ be the origin.
Denote by $L^1 _t$ and $L^2 _t$ the length of line segment $P_1 P_2$ and of $P_1 O ,$ respectively.
Then $L^1 _t <  L^2 _t .$
Since $f(z)=z^2$ for all $ z \in \partial M_s ,$  
we have 
\begin{equation}\label{negative K_f :0-22}
L^1 _t \approx \frac{t}{2} \mbox{ and }L^2 _t \approx t  \qquad \mbox{whenever }t \ll 1.
\end{equation}
Let $\hat{S}(P_1 , r) = S(P_1 , r) \cap f^{-1} (\Omega) .$
From the ACL-property of Sobolev functions and H\"older's inequality, we have that
\begin{equation}\label{negative K_f :0-1}
\osc_{\hat{S}(P_1 , r)} u \le \int_{\hat{S}(P_1 , r)} |D u| \, d s \le (2 \pi r)^{\frac{p-1}{p}} \left(\int_{\hat{S}(P_1 , r)} |D u|^p \, d s \right)^{\frac{1}{p}}
\end{equation}
for any $p > 1$ and $\mathcal{L}^1$-a.e. $r \in (L^1 _t ,L^2 _t) .$ 
Since $\osc_{\hat{S}(P_1 , r)} u = 1$ for all $ r \in (L^1 _t , L^2 _t),$ we conclude from \eqref{negative K_f :0-1} that
\begin{equation}\label{negative K_f :-100}
\int_{\hat{S}(P_1 , r)} |D u|^p \, d s \gtrsim r^{1-p} \qquad \mathcal{L}^1 \mbox{-a.e. } r \in (L^1 _t ,L^2 _t).
\end{equation}
Let $A_t = f^{-1} (\Omega)\cap B(P_1 , L^2 _t) \setminus \overline{B(P_1 , L^1 _t)} .$
By Fubini's theorem and \eqref{negative K_f :0-22},
we deduce from \eqref{negative K_f :-100} that 
\begin{equation}\label{negative K_f :1}
\int_{A_t} |D u|^p = \int_{L^1 _t} ^{L^2 _t} \int_{\hat{S}(P_1 , r)} |D u|^p \, ds \, dr
\gtrsim \int_{L^1 _t} ^{L^2 _t} r ^{1-p}\, dr \approx t^{2-p}.
\end{equation}
Let $Q_t = f^{-1} (\tilde{Q} _t) .$ 
From \eqref{negative K_f :-1}, we have $|Du(z)|=0$
for all $ z \in A_t \setminus Q _t .$ 
We hence conclude from \eqref{negative K_f :1} that
\begin{equation}\label{negative K_f :2}
\int_{Q_t} |D u|^p \ge  \int_{Q_t \cap A_t} |D u|^p = \int_{A_t} |D u|^p  \gtrsim t^{2-p} 
\end{equation}
for any $p \ge 1.$

From \eqref{negative K_f :0-2}, \eqref{MFD :1} and H\"older's inequality, it follows that for any $p \in (0,2)$ 
\begin{align}\label{negative K_f :3}
\int_{Q_t} |D u|^p \le &\int_{Q_t}|D v \circ f|^p |D f|^p
\le  \int_{Q_t}|D v \circ f|^p J^{\frac{p}{2}}_f K^{\frac{p}{2}}_f \notag \\
\le & \left(\int_{Q_t} |D v \circ f|^2 J_f \right)^{\frac{p}{2}} \left(\int_{Q_t} K^{\frac{p}{2-p}}_f \right)^{\frac{2-p}{2}} \notag\\
\le & \left(\int_{\tilde{Q} _t} |D v |^2 \right)^{\frac{p}{2}} \left(\int_{Q_t} K^{\frac{p}{2-p}}_f \right)^{\frac{2-p}{2}},
\end{align}
where the last inequality is from Lemma \ref{lemma A}. 
From \eqref{negative K_f :-1}, we have that
\begin{align}\label{negative K_f :0}
\int_{\tilde{Q} _t } |D v (x_1, x_2)|^2 \, dx_1 \, dx_2 = & \left(\int_{-t^2} ^{- (t/2)^2 } \frac{dx}{(-x)^{s}} \right)^{-2} \int_{-t^2} ^{- (t/2)^2 } \int_{-|x _1|^s} ^{|x _1|^s} \frac{1}{(-x _1)^{2s}} \, dx_2 \, dx_1 \notag \\
\approx & \left( \int_{-t^2} ^{- (t/2)^2 } \frac{dx}{(-x)^{s}} \right)^{-1} \approx t^{2(s-1)}.
\end{align}
Let $q =p/(2-p).$ Then $q \in [1,+\infty)$ whenever $p \in [1,2).$ 
Combining \eqref{negative K_f :0}, \eqref{negative K_f :2} with \eqref{negative K_f :3} yields
\begin{equation}\label{negative K_f :4}
t^{2+2(1-s) q} \lesssim \int_{Q_t} K^q_f 
\end{equation}
for all $q \ge 1.$
We now consider the set $Q_t$ for $t =2^{-j}$ with $j \ge j_0$ for a fixed large $j_0 .$
Analogously to \eqref{K_f^-1: 10}, 
it follows from \eqref{negative K_f :4} that
\begin{equation}\label{negative K_f :5}
\sum_{j=j_0} ^{+\infty} 2^{2j((s-1)q-1)} \lesssim 
\sum_{j=j_0} ^{+\infty}  \int_{Q_{2^{-j}}} K^q _{f} \le 2\int_{B(0,1)} K^q _{f}.
\end{equation}
Whenever $s \ge 2,$ the sum in \eqref{negative K_f :5} diverges if $q \ge 1 .$ 
Whenever $s \in (1,2),$ the sum in \eqref{negative K_f :5} also diverges if $q \ge 1/(s-1).$ Hence $K_f \in L^q _{loc} (\mathbb{R}^2)$ is possible only when $q < \max \{1, 1/(s-1)\}.$
\end{proof}

In Lemma \ref{necessary K_f}, we obtained an estimate for those $q$ for which $K_f \in L^q _{\loc} .$
We continue with the additional assumption that $f \in W^{1,p} _{\loc}$ for some $ p > 1 .$

\begin{lemma}\label{theorem Df+K_f: necessary}
Let $\mathcal{E}_s$ be as in \eqref{mathcalE_s} with $s > 2 .$ 
If $f \in \mathcal{E}_s$, $f \in W^{1,p} _{\loc} (\mbr^2 , \mbr^2)$ for some $p>1$ and $K_f \in L^q _{\loc} (\mbr^2)$ for some $q \in (0,1),$ then $q < 3p/((2s-1)p+4-2s) .$
\end{lemma}

\begin{proof}
Let $f$ be a homeomorphism with the above properties. By \cite[Theorem 4.1]{Hencl 2006 Arch. Ration. Mech. Anal.} we have  
$f^{-1} \in W^{1,r} _{\loc} (\mathbb{R}^2)$ where 
\begin{equation*}
r = \frac{(q+1)p -2q}{p-q}.
\end{equation*}
Moreover
\begin{equation*}\label{|Df|+K_f: 0}
 r < \frac{2(s+1)}{2s-1} \Leftrightarrow q < \frac{3p}{(2s-1)p+4-2s}.
\end{equation*}
Hence the claim follows from Lemma \ref{negative part of f^-1}.
\end{proof}

\begin{remark}\label{lem 3.1}
Notice that in the proof of Lemma \ref{necessary K_f} we only care about the property of $f$ in a small neighborhood of the origin.
Let $t \ll 1.$ By modifying $\partial M_s \cap B(0,t),$ we may generalize Lemma \ref{necessary K_f}. For example, we modify $\partial M_{3/2} \cap B(0,t)$ such that its image under $f(z)=z^2$ is 
\begin{equation*}
\{(x,y) \in \mbr^2 : x \in [-2^{-j_0} ,0],\ y^2=c |x|^3\}
\end{equation*} 
where $c$ is a positive constant.
If $K_f \in L^q _{\loc} (\mbr^2)$ for some $q \ge 1,$ by the analogous arguments as for Lemma \ref{necessary K_f} we have $q <2 .$ Similarly, one may extend Lemma \ref{negative part of f^-1}, Lemma \ref{necessary K_f^-1} and Lemma \ref{theorem Df+K_f: necessary} to the above setting.
\end{remark}

\begin{lemma}\label{rmk inverse}
Let $\Delta_s$ be as in \eqref{definition of M_s and Delta_s} with $s >1 .$ 
Suppose that $f :\mathbb{R}^2 \rightarrow \mathbb{R}^2$ is a homeomorphism of finite distortion 
such that $f$ maps $\mathbb{D}$ conformally onto $\Delta_s .$
We have that
\begin{enumerate}
\item if $f^{-1} \in W^{1,p} _{\loc} (\mathbb{R}^2 , \mbr^2)$ for some $p \ge 1$ then $p < 2(s+1)/(2s-1) ,$
\item if $K_{f^{-1}} \in L^q _{loc} (\mathbb{R}^2 )$ for some $q \ge 1 $ then $q < (s+1)/(s-1),$
\item if $K_f \in L^q _{loc} (\mathbb{R}^2)$ for some $q \ge 1$ then $q < \max \{1, 1/(s-1)\} ,$
\item if $s>2$, $f \in W^{1,p} _{\loc} (\mbr^2 , \mbr^2)$ for some $p>1$ and $K_f \in L^q _{\loc}$ for some $q \in (0,1),$ then $q < 3p/((2s-1)p+4-2s) .$ 
\end{enumerate}
\end{lemma}

\begin{proof}
Let $g_s$ be as in \eqref{g_s}, and $h_s = z^2 \circ g_s .$  
Since $h_s : \mathbb{D} \rightarrow  \Delta_s$ is conformal, there is a M\"obius transformation
\begin{equation*}
m_s (z) = e^{i \theta} \frac{z-a}{1-\bar{a} z}\qquad \mbox{where } \theta \in [0, 2\pi] \mbox{ and } |a|<1
\end{equation*}
such that $f(z)= h_s \circ m_s (z) $ for all $z \in \mbd .$ Since $m_s : \mbs^1 \rightarrow \mbs^1$ is a bi-Lipschitz mapping, by \cite[Theorem A]{Tukia 1980 Ann. Acad. Sci. Fenn. Ser. A I Math.} there is a bi-Lipschitz mapping $m^c _s : \mathbb{D}^c \rightarrow \Delta^c _s $ such that $m^c _s |_{\mathbb{S}^1} =m_s .$
Define
\begin{equation}\label{mathfrakM_s }
\mathfrak{M}_s (z) =
\begin{cases}
m_s (z) & z \in \overline{\mathbb{D}} , \\
m^c _s (z) & z \in \mathbb{D}^c .
\end{cases}
\end{equation} 
Then $\mathfrak{M}_s :\mathbb{R}^2 \rightarrow \mathbb{R}^2$ is a bi-Lipschitz, orientation-preserving mapping.
Let $G_s$ be as in \eqref{rmk inverse: 0}. Define
\begin{equation*}
F= f \circ \mathfrak{M}^{-1} _s \circ G^{-1} _s : \mbr^2 \rightarrow \mbr^2 .
\end{equation*}
Lemma \ref{composition} implies that $F \in \mathcal{E}_s ,$ where $\mathcal{E}_s$ is from \eqref{mathcalE_s}. From Lemma \ref{lemma C} and Lemma \ref{lemma B}, it follows that  
\begin{equation}\label{rmk inverse: 3}
\mbox{both } f^{-1} \mbox{ and } F^{-1} \mbox{ are differentiable }\mathcal{L}^2 \mbox{-a.e. on } \mathbb{R}^2 .
\end{equation}
Since 
\begin{align*}
\frac{|f^{-1} (z_1) - f^{-1} (z_2)|}{|z_1 -z_2|} = 
& \frac{|F^{-1} (z_1 ) - F^{-1} (z_2 )|}{|z_1- z_2 |}  \frac{|(G^{-1} _s (F^{-1} (z_1))-(G^{-1} _s (F^{-1} (z_2)) |}{|F^{-1} (z_1 ) - F^{-1} (z_2 )|} \times \\
& \times  \frac{|\mathfrak{M}^{-1}_s (G^{-1} _s \circ F^{-1} (z_1)) - \mathfrak{M}^{-1}_s (G^{-1} _s \circ F^{-1} (z_2))|}{|G^{-1} _s \circ F^{-1} (z_1) - G^{-1} _s \circ F^{-1} (z_2) |}
\end{align*}
for all $z_1, z_2 \in \mathbb{R}^2 $ with $z_1 \neq z_2 ,$ by \eqref{rmk inverse: 3} and the bi-Lipschitz properties of $G^{-1} _s $ and $\mathfrak{M}^{-1} _s$ we have that
\begin{equation}\label{rmk inverse: 4}
|Df^{-1} (z)| \approx |DF^{-1} (z)|,
\end{equation}
\begin{equation}\label{rmk inverse: 5}
\max_{\theta \in [0,2\pi]} |\partial_{\theta} f^{-1} (z)| \approx \max_{\theta \in [0,2\pi]} |\partial_{\theta} F^{-1} (z)|,\ \min_{\theta \in [0,2\pi]} |\partial_{\theta} f^{-1} (z)| \approx \min_{\theta \in [0,2\pi]} |\partial_{\theta} F^{-1} (z)| 
\end{equation}
for $\mathcal{L}^2$-a.e. $z \in \mathbb{R}^2 .$
If $f^{-1} \in W^{1,p}_{\loc} $ for some $p \ge 1 ,$
Lemma \ref{necessary K_f^-1} together with \eqref{rmk inverse: 7} gives $p < 2(s+1)/(2s-1).$
By \eqref{rmk inverse: 5} and \eqref{MFD :3} we have that
\begin{equation}\label{rmk inverse: 7}
 K_{f^{-1}} (z) \approx  K_{F^{-1}} (z) \qquad \mathcal{L}^2 \mbox{-a.e. } z \in \mathbb{R}^2 .
\end{equation} 
If $K_{f^{-1}} \in L^q _{loc} (\mathbb{R}^2)$ for some $q \ge 1 ,$
combining \eqref{rmk inverse: 4} and Lemma \ref{negative part of f^-1} then yields $q < (s+1)/(s-1).$

By Lemma \ref{composition} and and Lemma \ref{lemma B}, we have that 
\begin{equation}\label{rmk2 :4}
\mbox{both } f \mbox{ and } F \mbox{ are differentiable }\mathcal{L}^2 \mbox{-a.e. on } \mathbb{R}^2 . 
\end{equation}
From \cite[Corollary 3.7.6]{Astala 2009}, $G_s \circ \mathfrak{M}_s$ satisfies Lusin ($N$) and $(N^{-1})$ conditions.
Since 
\begin{align*}
\frac{|f(z_1 ) -f(z_2)|}{|z_1 -z_2|} = & \frac{|F (G_s \circ \mathfrak{M}_s (z_1)) -F (G_s \circ \mathfrak{M}_s (z_2))|}{|G_s \circ \mathfrak{M}_s (z_1) -G_s \circ \mathfrak{M}_s (z_2)|} \frac{|G_s (\mathfrak{M}_s (z_1)) -G_s (\mathfrak{M}_s (z_2))|}{|\mathfrak{M}_s (z_1) -\mathfrak{M}_s (z_2)|} \times \\
& \times  \frac{|\mathfrak{M}_s (z_1) -\mathfrak{M}_s (z_2)|}{|z_1 -z_2|} 
\end{align*}
for all $z_1 , z_2 \in \mathbb{R}^2$ with $z_1 \neq z_2, $ from \eqref{rmk2 :4} and the bi-Lipschitz properties of $G_s$ and $\mathfrak{M}_s$ we have that
\begin{equation}\label{rmk2 :5}
|Df(z)|  \approx |DF (G_s \circ \mathfrak{M}_s (z))|,
\end{equation}
\begin{equation}\label{rmk2 :6}
\max_{\theta \in [0,2\pi]} |\partial_{\theta} f (z)| \approx \max_{\theta \in [0,2\pi]} |\partial_{\theta} F (G_s \circ \mathfrak{M}_s (z))|,
\end{equation}
\begin{equation}\label{rmk2 :7}
 \min_{\theta \in [0,2\pi]} |\partial_{\theta} f (z)| \approx \min_{\theta \in [0,2\pi]} |\partial_{\theta} F (G_s \circ \mathfrak{M}_s (z))|
\end{equation}
for $\mathcal{L}^2$-a.e. $z \in \mathbb{R}^2 .$
By \eqref{MFD :3}, \eqref{rmk2 :6} and \eqref{rmk2 :7} we have that
\begin{equation}\label{rmk2 :8}
K_{f} (z) \approx K_F (G_s \circ \mathfrak{M}_s (z))\qquad  \mathcal{L}^2 \mbox{-a.e. } z \in \mathbb{R}^2 .
\end{equation}
Via the same reasons as for \eqref{composition: 7}, we have that
\begin{equation}\label{rmk2 :8-1}
J_{G_s \circ \mathfrak{M}_s}(z) \approx 1 \qquad \mathcal{L}^2 \mbox{-a.e. } z \in \mathbb{R}^2 .
\end{equation}
By \eqref{rmk2 :8-1} and Lemma \ref{lemma A}, we derive from \eqref{rmk2 :8} that 
\begin{align}\label{rmk2 :9}
\int_{A} K^q _{f} (z)\, dz = & \int_{A} K^q _{F} (G_s \circ \mathfrak{M}_s(z)) \frac{J_{G_s \circ \mathfrak{M}_s}(z)}{J_{G_s \circ \mathfrak{M}_s(z)}} \, dz \notag \\
\approx & \int_{A} K^q _{F} (G_s \circ \mathfrak{M}_s(z)) J_{G_s \circ \mathfrak{M}_s} (z)\, dz = \int_{G_s \circ \mathfrak{M}_s (A)} K^q _{F} (w)\, dw 
\end{align}
for any $q \ge 0$ and any compact set $A \subset \mathbb{R}^2 .$
By \eqref{rmk2 :5} and Lemma \ref{lemma A}, we obtain that
\begin{align}\label{rmk2 :10}
\int_{A} |Df(z)|^p = & \int_{A} |D F (G_s \circ \mathfrak{M}_s(z))|^p \frac{J_{G_s \circ \mathfrak{M}_s}(z)}{J_{G_s \circ \mathfrak{M}_s } (z)} \, dz \notag \\
\approx &  \int_{A} |D F (G_s \circ \mathfrak{M}_s(z))|^p J_{G_s \circ \mathfrak{M}_s} (z) \, dz = \int_{G_s \circ \mathfrak{M}_s (A)} |D F|^p (w) \, dw
\end{align}
for any $p \ge 0 .$
If $K_f \in L^q _{loc} (\mathbb{R}^2)$ for some $q \ge 1 ,$ Lemma \ref{necessary K_f} together with \eqref{rmk2 :9} gives that $q < \max \{1, 1/(s-1)\} .$
If $f \in W^{1,p} _{\loc}$ and $K_f \in L^q _{\loc}$ for some $p>1$ and some $q \in (0,1) ,$ combining Lemma \ref{theorem Df+K_f: necessary} with \eqref{rmk2 :10} then implies $q < 3p/((2s-1)p+4-2s) .$
\end{proof}

A result related to Lemma \ref{rmk inverse} ($3$) appeared in \cite[Theorem 4.4]{Guo 2014 Publ. Mat.}.

\section{Proof of Theorem \ref{theorem K_f}}

\subsection{$\mathcal{F}_s (f)\neq \emptyset$}\label{existence}

\begin{proof}
Let $g : \mathbb{D} \rightarrow \Delta_s$ be a conformal mapping with $s>1 .$
Analogously to \eqref{mathfrakM_s }, there is a bi-Lipschitz mapping $\mathfrak{M}_s : \mathbb{R}^2 \rightarrow \mbr^2 .$
Let $G_s$ be as in \eqref{rmk inverse: 0} and $\mathcal{E}_s$ be defined in \eqref{mathcalE_s}.
If $E \in \mathcal{E}_s ,$ by Lemma \ref{composition} we have $E \circ G_s \circ \mathfrak{M}_s \in \mathcal{F}_s (g) .$
We now divide the construction of $E$ into two steps:
Step $1$ deals with the construction in a neighborhood of the cusp point, see FIGURE \ref{neig};
Step $2$ gives the construction on the domain away from the cusp point. 

\bigskip

  \item [\textsf{Step} 1:]
Fix $ s >1,$ and define 
\begin{equation}\label{proof of f^-1: -1}
\eta (x) =\sqrt{x} (1+x^{2(s-1)})^{\frac{1}{4}}\qquad \mbox{ for all } x>0 .
\end{equation}
Then 
\begin{equation}\label{proof of f^-1: 0}
\eta' (x) = \frac{(1+x^{2(s-1)})^{\frac{1}{4}}}{2 \sqrt{x}} \left( 1+ \frac{(s-1) x^{2s-2}}{1+x^{2(s-1)}}\right).
\end{equation}
For a given $ t \ll 1,$ let 
\begin{equation}\label{proof of f^-1: 1}
L^1 _t = \eta( ( t/2 )^2) ,\ L^2 _t = \eta(t^2) \mbox{ and } \sigma_t  =L^2 _t -L^1 _t.
\end{equation}
Then $L^1 _t \approx t/2,\ L^2 _t \approx t$ and $\sigma_t \approx t/2 $ whenever $t \ll 1.$
Set
\begin{equation}\label{def f_1}
Q_t = \overline{B(0,L^2 _t)} \setminus ( B(0,L^1 _t) \cup M_s) \mbox{, and } f_1 (x, y)=x e^{i y} \quad \forall x \ge 0 \mbox{ and } y \in [0,2\pi].
\end{equation}
Let $\ell(r)$ be the length of $f^{-1}_1 (Q_t) \cap \{(x,y)\in \mathbb{R}^2 : x=r\}.$ 
Define
\begin{equation}\label{def f_2}
f_2 (r, \theta) = \left(r, \frac{\sigma_t }{\ell(r)} (\pi -\theta) \right) \qquad \forall (r, \theta) \in f^{-1}_1 (Q_t).
\end{equation}
Since $\partial M_s$ is mapped onto $\partial \Delta_s$ by $z^2,$
we have that
\begin{equation}\label{proof of f^-1: 2}
\ell(r) = \pi + \arctan \tau^{2(s-1)} 
\mbox{ and } r= \eta (\tau ^2)
\end{equation}
for all $\tau \in (t/2 ,t).$
Then $ \ell(r) \approx \pi$ and $r \approx \tau $ whenever $\tau \ll 1.$
From \eqref{proof of f^-1: 0}, it follows that $\frac{\partial r}{\partial \tau} \approx 1 .$
Together with $\frac{\partial \ell}{\partial \tau} \approx \tau^{2s-3},$ 
we have that
\begin{equation}\label{proof of f^-1: 3}
\frac{\partial \ell (r)}{\partial r} \approx r^{2s-3} \qquad \mbox{ for all } r \ll 1.
\end{equation}

Denote $R_t = f_2 \circ f^{-1} _{1} (Q_t).$ Then $R_t = [L^1 _t ,L^2 _t] \times [-\sigma_t /2 , \sigma_t /2] .$
Combining \eqref{def f_1} with \eqref{def f_2} implies
\begin{equation*}\label{proof of f^-1: 3-1}
f_1 \circ f^{-1} _2 (x,y) = \left(-x \cos\frac{\ell(x) y}{\sigma_t } , x \sin \frac{\ell(x) y}{\sigma_t } \right) \quad \forall (x,y) \in R_t .
\end{equation*}
Therefore 
\begin{equation}\label{D f_1 o f^-1 _2}
D f_1 \circ f^{-1} _2 (x,y)=
\begin{bmatrix}
-\cos\frac{\ell(x) y}{\sigma_t } + \frac{xy \ell'(x) }{\sigma_t } \sin\frac{\ell(x) y}{\sigma_t }
& \frac{x \ell(x) }{\sigma_t } \sin\frac{\ell(x) y}{\sigma_t } \\
 \sin \frac{\ell(x) y}{\sigma_t } + \frac{xy \ell'(x) }{\sigma_t } \cos\frac{\ell(x) y}{\sigma_t }
 & \frac{x \ell(x) }{\sigma_t } \cos\frac{\ell(x) y}{\sigma_t }
\end{bmatrix}.
\end{equation}
By \eqref{proof of f^-1: 1}, \eqref{proof of f^-1: 2} and \eqref{proof of f^-1: 3}, we deduce from \eqref{D f_1 o f^-1 _2} that
\begin{equation}\label{proof of f^-1: 4}
|D f_1 \circ f^{-1} _2 (x,y)| \lesssim 1
\mbox{ and } J_{f_1 \circ f^{-1} _2 }(x,y) = -\frac{x \ell(x) }{\sigma } \approx -1
\end{equation}
for all $t \ll 1 $ and each $(x,y) \in R_t .$
Since $K_{f_1 \circ f^{-1} _2} \ge 1, $
from \eqref{proof of f^-1: 4} we have  
\begin{equation}\label{proof of f^-1: 5}
K_{f_1 \circ f^{-1} _2} \approx 1 .
\end{equation}
By \eqref{proof of f^-1: 4} again we have that
\begin{equation}\label{proof of f^-1: 6}
|D f_2 \circ f^{-1} _1| =\frac{|adj D f_1 \circ f^{-1} _2|}{|J_{f_1 \circ f^{-1} _2 }|} \approx |D f_1 \circ f^{-1} _2 | \lesssim 1
\mbox{ and }
J_{f_2 \circ f^{-1} _1} = \frac{1}{J_{f_1 \circ f^{-1} _2 }} \approx -1.
\end{equation}
Analogously to \eqref{proof of f^-1: 5}, we have that
\begin{equation}\label{proof of f^-1: 7}
K_{f_2 \circ f^{-1} _1} (x,y) \approx 1 \qquad \forall t \ll 1 \mbox{ and } \forall (x,y) \in Q_t . 
\end{equation}

\begin{figure}[htbp]
\centering \includegraphics[bb=135 388 572 726,clip=true, scale=0.8]{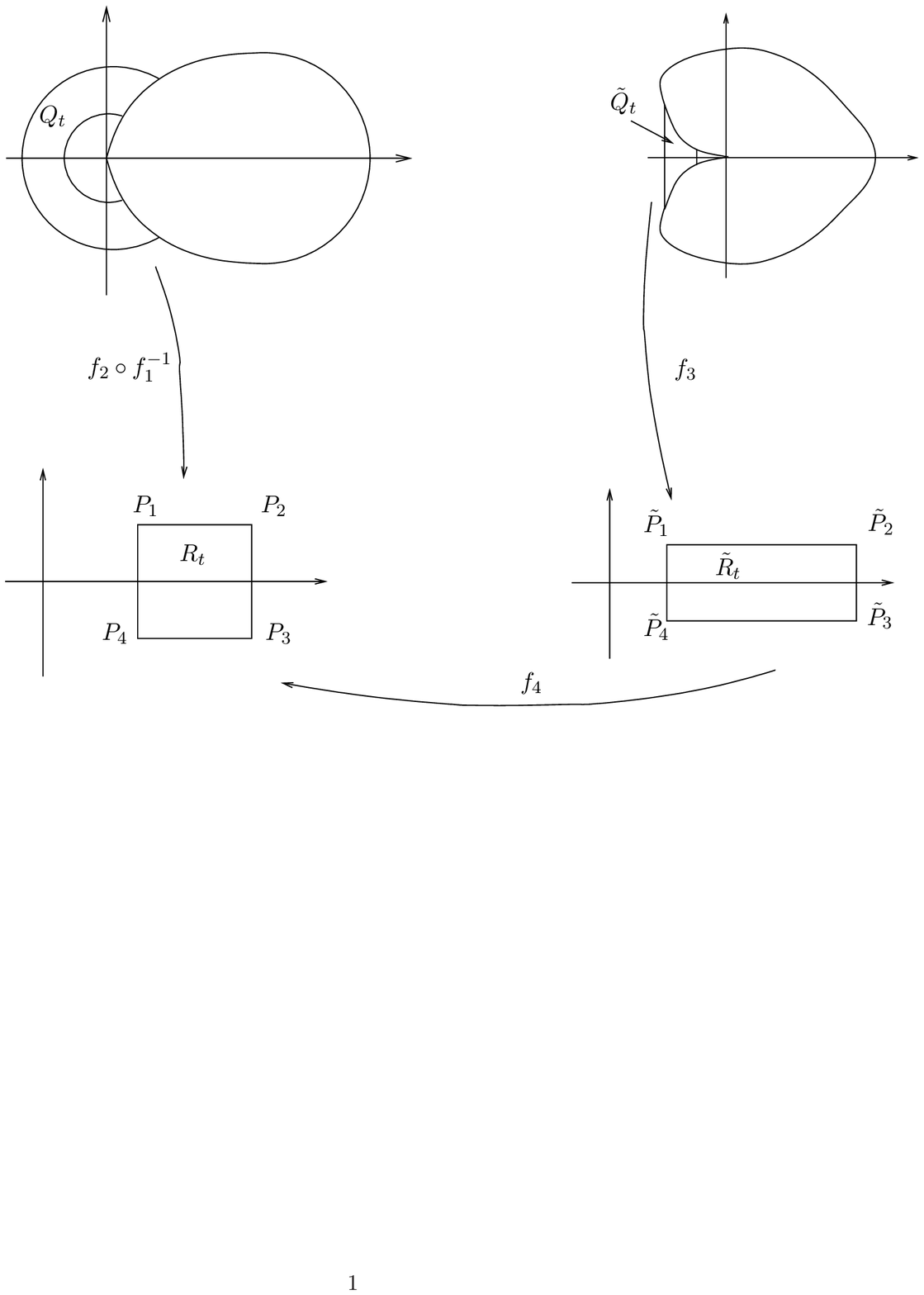}
\caption{The construction $f^{-1} _3 \circ f^{-1} _4 \circ f_2 \circ f^{-1} _1 :Q_t \rightarrow \tilde{Q}_t$}
\label{neig}
\end{figure}

Let
\begin{equation*}
\tilde{Q}_t = \{(x,y)\in \mathbb{R}^2 : x \in [-t^2, -(t/2)^2] ,\ |y| \le |x|^s\} .
\end{equation*}
Define 
\begin{equation*}\label{proof of f^-1: 7-1}
f_3 (u,v) = \left(-u, \frac{t^{2s}}{(-u)^s} v \right) \qquad \forall (u,v) \in \tilde{Q}_t.
\end{equation*}
Then $f_3$ is diffeomorphic and
\begin{equation}\label{proof of f^-1: -8}
Df_3(u,v)=
\begin{bmatrix}
-1 & 0 \\
 \frac{s t^{2s}}{(-u)^{s+1}} v & \frac{t^{2s}}{(-u)^s}
\end{bmatrix}.
\end{equation}
From \eqref{proof of f^-1: -8} we have that
\begin{equation}\label{proof of f^-1: 8}
|Df_3| \lesssim 1 \mbox{ and } J_{f_3} \approx -1 \qquad \forall (u,v) \in \tilde{Q}_t.
\end{equation}
Analogously to \eqref{proof of f^-1: 5}, we have that
\begin{equation}\label{proof of f^-1: 9}
K_{f_3}(u,v) \approx 1 \qquad \forall t \ll 1 \mbox{ and }  \forall (u,v) \in \tilde{Q}_t .
\end{equation}
Let $\tilde{R}_t = f_3 (\tilde{Q}_t).$ Then $\tilde{R}_t = [(t/2)^2 , t^2] \times [-t^{2s} ,t^{2s}].$
The same reasons as for \eqref{proof of f^-1: 6} and \eqref{proof of f^-1: 7} imply that
\begin{equation}\label{proof of f^-1: 10}
|Df^{-1} _3 (x,y)| \lesssim 1,\ J_{f^{-1} _3} (x,y)\approx -1 \mbox{ and } K_{f^{-1} _3} (x,y)\approx 1 
\end{equation}
for all $t \ll 1$ and $(x,y) \in \tilde{R}_t .$

Denote by 
$ P_1, P_2, P_3, P_4 $ and $\tilde{P}_1 , \tilde{P}_2, \tilde{P}_3 ,\tilde{P}_4$
the four vertices of $\tilde{R}_t$ and $R_t ,$ respectively.
Then 
\begin{equation*}
P_1 =(L^1 _t , \frac{\sigma_t }{2}),\ P_2 =(L^2 _t , \frac{\sigma_t}{2}),\ 
P_3 =(L^2 _t , -\frac{\sigma_t }{2}),\ P_4 =(L^1 _t , -\frac{\sigma_t }{2}) 
\end{equation*}
and
\begin{equation*}
\tilde{P}_1 = ((t/2)^2 ,  t^{2s}),\ \tilde{P}_2 = (t^2 , t^{2s}),\ \tilde{P}_3 = (t^2 , -t^{2s}),\ \tilde{P}_4 = ((t/2)^2 , -t^{2s}). 
\end{equation*}
Since $\partial M_s$ is mapped onto $\partial \Delta_s$ by $z^2 ,$ the line segment $\tilde{P}_1 \tilde{P}_2$ is mapped onto $P_1 P_2$ by 
\begin{equation*}
(u,t^{2s}) \mapsto \left(\eta (u),\frac{\sigma_t}{2} \right)\qquad \forall u \in [(t/2)^2 ,t^2],
\end{equation*}
and the line segment $\tilde{P}_4 \tilde{P}_3$ is mapped onto $P_4 P_3$ by 
\begin{equation*}
(u,-t^{2s}) \mapsto \left(\eta (u), -\frac{\sigma_t }{2} \right)\qquad \forall u \in [(t/2)^2 ,t^2].
\end{equation*}
Define
\begin{equation}\label{proof of f^-1: 10-0}
f_4 (u,v) = \left(\eta (u), \frac{\sigma_t }{2 t^{2s}}v \right) \qquad \forall (u,v) \in \tilde{R}_t.
\end{equation}
Then $f_4$ is a diffeomorphism from $\tilde{R}_t$ onto $R_t $
and
\begin{equation}\label{proof of f^-1: 10-01}
D f_4 (u,v)=
\begin{bmatrix}
\eta' (u) & 0 \\
0 & \frac{\sigma_t }{2 t^{2s}}
\end{bmatrix}.
\end{equation}
By \eqref{proof of f^-1: 0} and \eqref{proof of f^-1: 1} we have that $\eta' (u) \approx t^{-1}$ and $\frac{\sigma_t }{2 t^{2s}} \approx t^{1-2s} $ whenever $t \ll 1$ and $(u,v) \in \tilde{R}_t .$ It follows from \eqref{proof of f^-1: 10-01} that
\begin{equation}\label{proof of f^-1: 10-1}
|Df_4 (u,v)| \approx t^{1-2s} \mbox{ and } J_{f_4} (u,v) \approx t^{-2s}
\end{equation}
for all $t \ll 1$ and all $(u,v) \in \tilde{R}_t .$  Then 
\begin{equation}\label{proof of f^-1: 11}
K_{f_4} (u,v)= \frac{|Df_4 (u,v)|^2}{J_{f_4} (u,v)} \approx t^{2-2s} \qquad \forall t \ll 1 \mbox{ and }  (u,v) \in \tilde{R}_t . 
\end{equation}
The same reasons as for \eqref{proof of f^-1: 6} and \eqref{proof of f^-1: 7} imply that
\begin{equation}\label{proof of f^-1: 12}
|Df^{-1} _4 (x,y)| \approx t,\ J_{f^{-1} _4} (x,y) \approx t^{2s} \mbox{ and } K_{f^{-1} _4} (x,y) \approx t^{2-2s}
\end{equation}
for all $t \ll 1$ and all $(x,y) \in R_t .$

Define
\begin{equation*}
F_t=f^{-1} _3 \circ f^{-1}_4 \circ f_2 \circ f^{-1} _{1}  .
\end{equation*}
Then $F_t$ is a diffeomorphism from $Q_t$ onto $\tilde{Q}_t .$  Therefore   
\begin{equation*}
DF_t (z)= Df^{-1} _3 (f^{-1}_4 \circ f_2 \circ f^{-1} _{1} (z)) Df^{-1}_4 ( f_2 \circ f^{-1} _{1} (z)) D (f_2 \circ f^{-1} _{1}) (z)
\end{equation*}  
for all $z \in Q_t  .$
From \eqref{proof of f^-1: 10}, \eqref{proof of f^-1: 12} and \eqref{proof of f^-1: 6} it then follows that 
\begin{align}\label{proof of f^-1: 12-1}
\int_{Q_t} |DF_t |^p \, dz\le & \int_{Q_t}  |D f^{-1} _3 (f^{-1} _4 \circ f_2 \circ f^{-1} _1) |^p |Df^{-1} _4 (f_2 \circ f^{-1} _1 )|^p |D f_2 \circ f^{-1} _1 |^p \, dz\notag\\
\lesssim & t^{p} \mathcal{L}^2(Q_t) \approx t^{2+p}
\end{align}
for any $p \ge 0.$
By Lemma \ref{lemma A} we have that
\begin{align}\label{proof of f^-1: 12-2}
\int_{Q_t} |J_{F_t} (z)| \, dz = & \int_{Q_t}  |J_{ f^{-1} _3} (f^{-1} _4 \circ f_2 \circ f^{-1} _1 (z))| |J_{f^{-1} _4 }(f_2 \circ f^{-1} _1 (z))| |J_{f_2 \circ f^{-1} _1 }(z)| \, dz\notag \\
\le & \int_{f_2 \circ f^{-1} _1 (Q_t)} |J_{ f^{-1} _3} (f^{-1} _4 )| |J_{f^{-1} _4 } | \notag\\
\le & \int_{f^{-1} _4  \circ f_2 \circ f^{-1} _1 (Q_t)} |J_{ f^{-1} _3} | \le  \mathcal{L}^2 (\tilde{Q}_t ).
\end{align}
For a fixed large $j_0,$
we now consider the set $Q_t$ with $t=2^{-j}$ for all $j \ge j_0 .$
Define
\begin{equation}\label{definition of E_1}
E_1 = \sum_{j=j_0} ^{+\infty} F_{2^{-j}} \chi_{Q_{2^{-j}}}.
\end{equation}
Denote $\Omega_1 = \cup_{j=j_0} ^{+\infty} Q_{2^{-j}} \mbox{ and } \tilde{\Omega}_1 = \cup_{j=j_0} ^{+\infty} \tilde{Q}_{2^{-j}}.$
Then $E_1$ is a homeomorphism from $\Omega_1$ onto $\tilde{\Omega}_1 ,$ and satisfies \eqref{MFD :1} for $E_1$ on $\mathcal{L}^2$-a.e. $\Omega_1 .$ 
In order to prove that $E_1$ has finite distortion on $\Omega_1 ,$ it thus suffices to prove that $E_1 \in W^{1,1} _{\loc} (\Omega_1)$ and $J_{E_1} \in L^1 _{\loc} (\Omega_ 1).$ Actually, 
from \eqref{proof of f^-1: 12-1} and \eqref{proof of f^-1: 12-2} we have that
\begin{equation}\label{proof of f^-1: 12-3}
\int_{\Omega_1} |D E_1|^p =\sum_{j=j_0} ^{+\infty} \int_{Q_{2^{-j}}} |DF_{2^{-j}} (z)|^p \, dz \lesssim \sum_{j=j_0} ^{+\infty} 2^{-j(2+p)} <\infty
\end{equation}
and 
\begin{equation}\label{proof of f^-1: 12-4}
\int_{\Omega_1} |J_{E_1}|= \sum_{j=j_0} ^{\infty} \int_{Q_{2^{-j}}} |J_{F_{2^{-j}}} | \le \sum_{j=j_0} ^{\infty} \mathcal{L}^2(\tilde{Q}_{2^{-j}} ) = \mathcal{L}^2(\tilde{\Omega}_1) <\infty
\end{equation}
for all $p \ge 1 .$ 

\bigskip
  \item [\textsf{Step} 2:]
Denote
\begin{equation*}
\Omega_2 = M_s ^c \setminus \Omega_1 \mbox{ and } \tilde{\Omega}_2 = \Delta_s ^c \setminus \tilde{\Omega}_1 .
\end{equation*}
Notice that both $\partial \Omega_2$ and $\partial \tilde{\Omega}_2$
are piecewise smooth Jordan curves with non-zero angles at the two corners. Therefore both $\partial \Omega_2$ and $\partial \tilde{\Omega}_2$ are chord-arc curves.
By \cite{Jerison 1982 Math. Scand.} there are bi-Lipschitz mappings 
\begin{equation}\label{H_1 and H_2}
H_1 : \mathbb{R}^2 \rightarrow \mathbb{R}^2 \mbox{ and } H_2 : \mathbb{R}^2 \rightarrow \mathbb{R}^2
\end{equation}
such that 
$H_1 (\mathbb{S}^1) = \partial \Omega_2$ and $H_2 (\mathbb{S}^1) = \partial \tilde{\Omega}_2 .$
Define
\begin{equation*}\label{h}
h(z)=
\begin{cases}
E_1 (z) & \forall z \in \partial \Omega_2 \cap  \partial \Omega_1 ,\\
z^2 & \forall z \in \partial \Omega_2 \cap \partial M_s.
\end{cases}
\end{equation*}
Then $h$ is a bi-Lipschitz mapping in terms of the arc lengths.
By the chord-arc properties of both $\partial \Omega_2$ and $\partial \tilde{\Omega}_2 ,$ we have that $h$ is also a bi-Lipschitz mapping with respect to the Euclidean distances. 
Taking \eqref{H_1 and H_2} into account, we conclude that
$H^{-1} _2 \circ h \circ H_1 : \mathbb{S}^1 \rightarrow \mathbb{S}^1 $ is a bi-Lipschitz mapping.  
By \cite[Theorem A]{Tukia 1980 Ann. Acad. Sci. Fenn. Ser. A I Math.} there is then a bi-Lipschitz mapping 
\begin{equation}\label{H_0}
H : \mathbb{R}^2 \rightarrow \mathbb{R}^2
\end{equation}
such that $H |_{\mathbb{S}^1} = H^{-1} _2 \circ h \circ H_1 .$
Define
\begin{equation}\label{E_2}
E_2 = H _2 \circ H \circ H^{-1} _1 .
\end{equation}
By \eqref{H_1 and H_2} and \eqref{H_0}, we have that
$E_2$
is a bi-Lipschitz extension of $h .$ 
Furthermore since $\deg_{M_s} (h, w) =1 ,$ we obtain that $E_2$ is orientation-preserving. Hence $E_2$ is a quasiconformal mapping. The same reasons as for \eqref{composition: 6} and \eqref{composition: 7} imply
\begin{equation}\label{key tool: 1}
|DE_2 (z)|,\ K_{E_2} (z) \mbox{ and } J_{E_2} (z) \mbox{ are bounded from both above and below}
\end{equation}
for $\mathcal{L}^2$-a.e. $z \in \mathbb{R}^2, $
and 
\begin{equation}\label{key tool: 2}
|DE^{-1} _2 (w)|,\ K^{-1} _{E_2} (w) \mbox{ and } J^{-1} _{E_2} (w) \mbox{ are bounded from both above and below}
\end{equation}
for $\mathcal{L}^2$-a.e. $w \in \mathbb{R}^2 .$

Via \eqref{definition of E_1} and \eqref{E_2}, we define
\begin{equation}\label{definition of E}
E(x,y)=
\begin{cases}
E_1 (x,y)& \mbox{ for all }(x,y) \in \Omega_1, \\
E_2 (x,y)& \mbox{ for all }(x,y) \in \Omega_2 ,\\
(x^2 -y^2 , 2xy) & \mbox{ for all }(x,y) \in \overline{M_s}.
\end{cases}
\end{equation}
By the properties of $E_1$ and $E_2,$ we conclude that $E \in \mathcal{E}_s .$ 
\end{proof}

\subsection{\eqref{theorem |df|}, \eqref{theorem f^-1 : 1} and \eqref{theorem f^-1 : 2}}

\begin{proof}[Proof of \eqref{theorem |df|}]
Let $g :\mathbb{D} \rightarrow \Delta_s$ be conformal, where $\Delta_s$ is defined in \eqref{definition of M_s and Delta_s} with $s >1 .$
In order to prove \eqref{theorem |df|}, it is enough to construct
$f \in \mathcal{F}_s (g)$ such that $f \in W^{1,p} _{\loc} (\mathbb{R}^2 , \mbr^2)$ for all $p \ge 1 .$ 
let $E$ be as in \eqref{definition of E}. Then $E \in \mathcal{E}_s .$
By \eqref{proof of f^-1: 12-3}, \eqref{key tool: 1} and the fact that $E(z)=z^2$ for all $ z \in M_s,$ we obtain that $E \in W^{1,p} _{\loc} (\mathbb{R}^2 , \mbr^2)$ for all $p \ge 1 .$
Let $G_s$ be as in \eqref{rmk inverse: 0} and $\mathfrak{M}_s$ be as in \eqref{mathfrakM_s }.
By Lemma \ref{composition} and the analogous arguments as for \eqref{rmk2 :10}, we can define $ f = E \circ G_s \circ \mathfrak{M}_s .$ 
\end{proof}

\begin{proof}[Proof of \eqref{theorem f^-1 : 1}]
Let $g :\mathbb{D} \rightarrow \Delta_s$ be conformal, where $\Delta_s$ is defined in \eqref{definition of M_s and Delta_s} with $s >1 .$ 
In order to prove \eqref{theorem f^-1 : 1}, by Lemma \ref{rmk inverse} ($1$) it is enough to construct a mapping $f \in \mathcal{F}_s (g)$ such that $f^{-1} \in W^{1,p}_{\loc} (\mathbb{R}^2 , \mbr^2)$ for all $p < 2(s+1) /(2s-1).$
Let $G_s$ be as in \eqref{rmk inverse: 0} and $\mathfrak{M}_s$ be defined in \eqref{mathfrakM_s }.
If there is a mapping $E \in \mathcal{E}_s$ such that $E^{-1} \in W^{1,p}_{\loc} (\mathbb{R}^2 ,\mbr^2)$ for all $p < 2(s+1) /(2s-1),$
by Lemma \ref{composition} and analogous arguments as for \eqref{rmk inverse: 4} we can define $f = E \circ G_s \circ \mathfrak{M}_s .$  

Let $E$ be as in \eqref{definition of E}. 
Then $E \in \mathcal{E}_s .$
By \eqref{proof of f^-1: 8}, \eqref{proof of f^-1: 10-1} and 
\eqref{proof of f^-1: 4} we have that 
\begin{equation*}
|DF^{-1} _{2^{-j}} (w)| \le |D f_1 \circ f^{-1}_2 (f_4 \circ f_3 (w))| |Df_4  (f_3 (w))||Df_3 (w)|
\lesssim 2^{j(2s-1)}
\end{equation*}
for all $j \ge j_0$ and $\mathcal{L}^2$-a.e. $w \in \tilde{Q}_{2^{-j}} .$
Together with $\mathcal{L}^2(\tilde{Q}_{2^{-j}} )\approx 2^{-2j(s+1)},$ we hence obtain that 
\begin{equation}\label{int_widetilde Omega_1 |DE^-1 _1|^p}
\int_{\tilde{\Omega}_1} |DE^{-1}_1|^p =  \sum_{j =j_0} ^{+\infty} \int_{\tilde{Q}_{2^{-j}}} |DF^{-1} _{2^{-j}}|^p 
\lesssim  \sum_{j =j_0} ^{+\infty} 2^{-j(2(s+1)+p(1-2s))} <\infty
\end{equation}
for all $p < 2(s+1)/(2s-1).$
Since 
\begin{equation}\label{E^-1 on Delta_s : 2-1}
|DE^{-1} (u,v)| \lesssim (u^2 +v^2)^{-1/4} \qquad \forall (u,v) \in \Delta_s ,
\end{equation}
by a change of variables we have that 
\begin{equation}\label{E^-1 on Delta_s : 3-1}
\int_{\Delta_s} |DE^{-1} (w)|^p \, dw \lesssim \int_{0} ^{2 \pi} \int_{0} ^{1} r^{1-\frac{p}{2}} \, dr \, d\theta \approx \int_{0} ^{1} r^{1-\frac{p}{2}} \, dr <\infty
\end{equation}
for all $p < 2(s+1)/(2s-1) .$ 
By \eqref{key tool: 2}, \eqref{int_widetilde Omega_1 |DE^-1 _1|^p} and \eqref{E^-1 on Delta_s : 3-1}, we conclude that $E^{-1} \in W^{1,p}_{\loc} (\mathbb{R}^2 ,\mbr^2)$ for all $p < 2(s+1) /(2s-1).$
\end{proof}

\begin{proof}[Proof of \eqref{theorem f^-1 : 2}]
Let $g :\mathbb{D} \rightarrow \Delta_s$ be conformal, where $\Delta_s$ is defined in \eqref{definition of M_s and Delta_s} with $s >1 .$ 
In order to prove \eqref{theorem f^-1 : 2}, by Lemma \ref{rmk inverse} ($2$) it is enough to construct a mapping $f \in \mathcal{F}_s (g)$ such that $K_{f^{-1}} \in L^{q}_{\loc} (\mathbb{R}^2)$ for all $q < (s+1) /(s-1).$
Let $G_s$ be as in \eqref{rmk inverse: 0} and $\mathfrak{M}_s$ be as in \eqref{mathfrakM_s }.
If there is a mapping $E \in \mathcal{E}_s$ such that $K_{E^{-1}} \in L^{q}_{\loc} (\mathbb{R}^2)$ for all $q < (s+1) /(s-1),$
by Lemma \ref{composition} and analogous argument as for \eqref{rmk inverse: 7} we can define $f = E \circ G_s \circ \mathfrak{M}_s .$ 

Let $E$ be as in \eqref{definition of E}. 
Then $E \in \mathcal{E}_s .$
From \eqref{proof of f^-1: 5}, \eqref{proof of f^-1: 11} and \eqref{proof of f^-1: 9}, we have that
\begin{equation*}
K_{F^{-1} _{2^{-j}}} (w) = K_{f_1 \circ f^{-1} _2} (f_4 \circ f_3 (w)) K_{f_4} (f_3 (w)) K_{f_3} (w) \approx 2^{j(2s-2)}
\end{equation*}
for all $j \ge j_0 $ and $\mathcal{L}^2$-a.e. $w \in \tilde{Q}_{2^{-j}} .$ 
Together with $\mathcal{L}^2(\tilde{Q}_{2^{-j}} )\approx 2^{-j2(s+1)},$ we then obtain that
\begin{equation}\label{buzhi 1}
\int_{\tilde{\Omega}_1} K^q _{E^{-1}}=  \sum_{j =j_0} ^{+\infty} \int_{\tilde{Q}_{2^{-j}}} K^q _{F^{-1} _{2^{-j}}}
\lesssim  \sum_{j =j_0} ^{+\infty} 2^{2j[(s-1)q -(s+1)]} <\infty 
\end{equation}
for all $q < (s+1) /(s-1).$
By \eqref{key tool: 2}, \eqref{buzhi 1} and the fact that $E$ is conformal on $M_s ,$
we conclude that $K_{E^{-1}} \in L^{q}_{\loc} (\mathbb{R}^2)$ for all $q < (s+1) /(s-1) .$

\end{proof}

\subsection{\eqref{theorem K_f: 1}}\label{proof of delta_1}

\begin{proof}
Let $g :\mathbb{D} \rightarrow \Delta_s$ be conformal, where $\Delta_s$ is defined as \eqref{definition of M_s and Delta_s} with $s >1 .$
In order to prove \eqref{theorem K_f: 1}, via Lemma \ref{rmk inverse} ($3$) it is enough to construct a mapping $f \in \mathcal{F}_s (g)$ such that $K_{f} \in L^q _{\loc} $ for all $q < \max\{1, 1/(s-1)\} .$
Let $G_s$ be as in \eqref{rmk inverse: 0} and $\mathfrak{M}_s$ be as in \eqref{mathfrakM_s }. 
If $E \in \mathcal{E}_s$ such that $K_{E} \in L^q _{\loc} $ for all $q < \max\{1, 1/(s-1)\} ,$
by Lemma \ref{composition} and analogous arguments as for \eqref{rmk2 :9} we can define $f = E \circ G_s \circ \mathfrak{M}_s .$

Let $E$ be as in \eqref{definition of E}.
Then $E \in \mathcal{E}_s .$
From \eqref{proof of f^-1: 10}, \eqref{proof of f^-1: 12} and \eqref{proof of f^-1: 7}, it follows that 
\begin{equation*}
K_{F_{2^{-j}}} (z) = K_{f^{-1} _3} (f^{-1} _4 \circ f_2 \circ f^{-1} _{1} (z))
K_{f^{-1} _4} (f_2 \circ f^{-1} _{1} (z)) K_{f_2 \circ f^{-1} _{1}} (z) 
\approx  2^{2j(s-1)}
\end{equation*}
for all $j \ge j_0 $ and $\mathcal{L}^2$-a.e. $z \in Q_{2^{-j}}.$ Together with $\mathcal{L}^2(Q_{2^{-j}}) \approx 2^{-2j}$ we then have that
\begin{equation}\label{proof of K_f :0}
\int_{\Omega_1} K^q _{E} = \sum_{j=j_0} ^{+\infty} \int_{Q_{2^{-j}}} K^q _{F_{2^{-j}}} \approx  \sum_{j=j_0} ^{+\infty} 2^{2j(q(s-1) -1)} <\infty
\end{equation}
for all $q < 1/(s-1) .$
By \eqref{proof of K_f :0}, \eqref{key tool: 1} and the fact that $E$ is conformal on $M_s ,$ we conclude that $K_E \in L^q _{\loc} (\mathbb{R}^2)$ for all $q < 1/(s-1) .$ 
Therefore we have proved \eqref{theorem K_f: 1} whenever $s \in (1,2).$ 

We next consider the case $s \in [2 ,\infty ) .$ It is enough to construct a mapping $E \in \mathcal{E}_s$ such that $K_{E} \in L^q _{\loc} $ for all $q < 1.$ 
Except for redefining $f^{-1} _4 : R_t \rightarrow \tilde{R}_t $ as in \eqref{proof of f^-1: 10-0},
we follow all processes in Section \ref{existence} to define a new $E,$ see FIGURE \ref{126}.
Let $\alpha_t$ and $\beta_t$ be the length of sides of $\tilde{R}_t,$ and $\gamma_t$ be the length of a side of $R_t.$ Whenever $t \ll 1,$ we have that
\begin{equation}\label{proof of K_f :1}
\alpha_t = t^2 -(t/2)^2 \approx t^2,\ \beta_t =2t^{2s} \mbox{ and } \gamma_t =\eta(t^2) -\eta ((t/2)^2) \approx t.
\end{equation}
Let $\tilde{T}_0 =\tilde{Q}_1 \tilde{Q}_2 \tilde{Q}_3 \tilde{Q}_4 $ be the concentric square of $\tilde{R}_t $ with side length $\beta_t /2 .$
Set
\begin{equation}\label{delta_t}
\delta_t =\exp(- t^{-1}) \qquad \mbox{for } t >0
\end{equation}
and let $T_0 = Q_1 Q_2 Q_3 Q_4 $ be the concentric square of $R_t$ with side length $\gamma_t (1-2 \delta_t).$ 
We divide $R_t \setminus T_0$ into four isosceles trapezoids $T_1,\ T_2,\ T_3 $ and $T_4 .$ Similarly, we obtain isosceles trapezoids $\tilde{T}_1,\ \tilde{T}_2,\ \tilde{T}_3,\ \tilde{T}_4 $ from $\tilde{R}_t \setminus \tilde{T}_0 ,$ see FIGURE \ref{126}.

\begin{figure}[htbp]
\centering \includegraphics[bb=138 553 571 726, clip=true, scale=0.8]{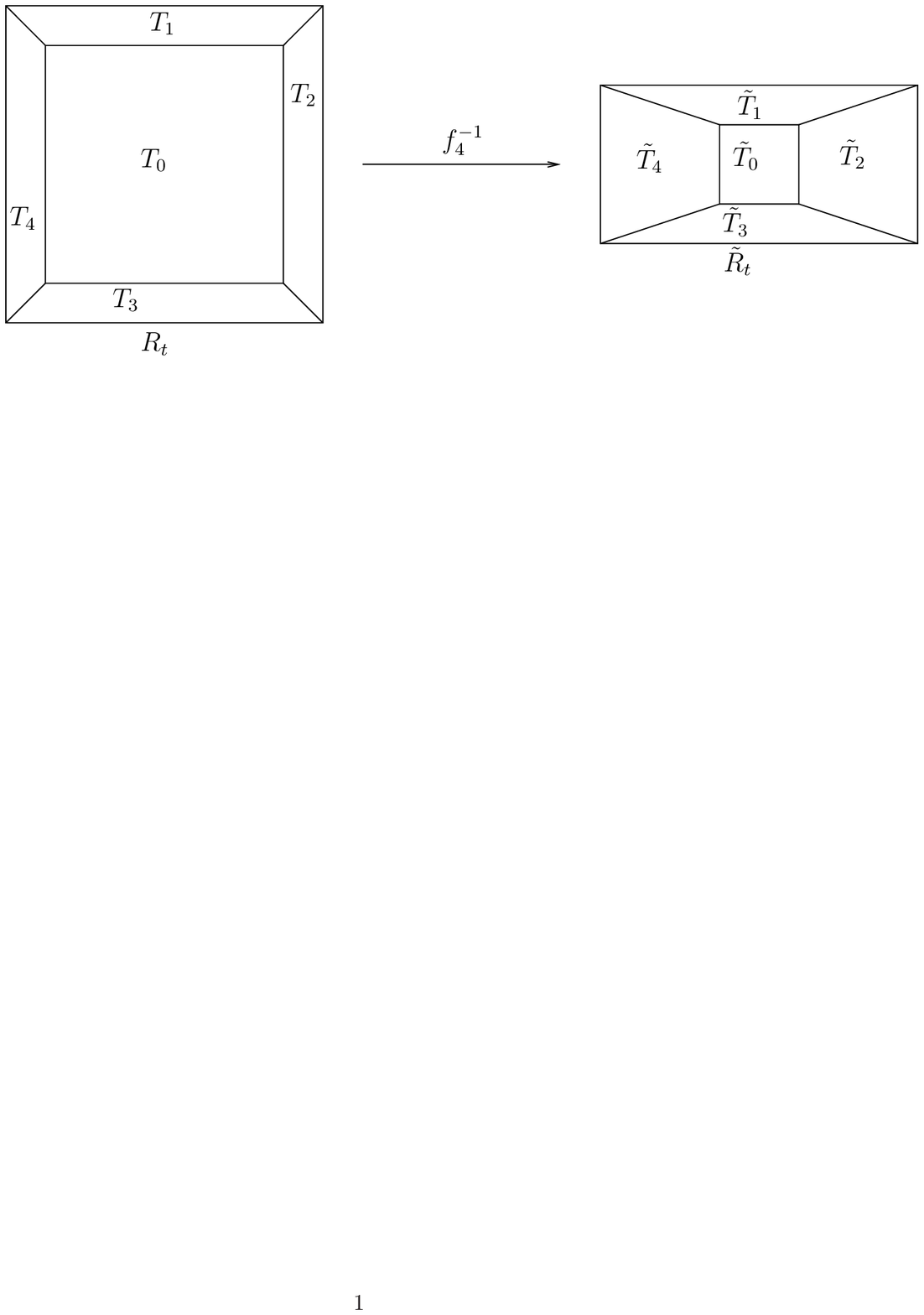}

\caption{The redefined $f^{-1} _4 :R_t \rightarrow \tilde{R}_t$}
\label{126}
\end{figure}

We first define a diffeomorphism from $T_1$ onto $\tilde{T}_1 .$ Define
\begin{equation}\label{proof of K_f :2}
A_2 (x,y) = \frac{\beta_t}{4 \delta_t \gamma_t} \left(y-\gamma_t \big(\frac{1}{2} -\delta_t \big) \right) +\frac{\beta_t}{4} \qquad \forall (x,y) \in T_1 .
\end{equation}
For a given $(x,y) \in T_1 ,$ let $(x_p, y) = P_1 Q_1 \cap \{(X,Y) \in \mathbb{R}^2 : Y=y\}$, $(\tilde{x}_p , A_2 ) = \tilde{P}_1  \tilde{Q}_1  \cap \{(X,Y)\in \mathbb{R}^2 : Y= A_2 (x,y) \}$, 
$\ell (y)$ be the length of $T_1 \cap \{(X,Y): Y= y\},$ and $\tilde{\ell} (y) $ be the length of $\tilde{T}_1 \cap \{(X,Y): Y= A_2\}.$
Denote $(P_1)_1$ by the first coordinate of $P_1 .$
Then
\begin{equation}\label{proof of K_f :4-0}
x_p = -y + \frac{\gamma_t}{2} + (P_1)_1 \mbox{ and }
\tilde{x}_p = \frac{2\alpha_t -\beta_t}{\beta_t} \left(\frac{\beta_t}{2} -A_2 \right) +(\tilde{P}_1)_1 ,
\end{equation}
\begin{equation}\label{proof of K_f :4-1}
\ell(y) = 2y \approx \gamma_t 
\mbox{ and } \tilde{\ell} (y) = \frac{4 \alpha_t -2 \beta_t}{\beta_t} A_2 (x,y) + \beta_t -\alpha_t \ge \frac{\beta_t}{2}.
\end{equation}
Let $u=\frac{\gamma_t}{\ell (y)} (x-x_p) + (P_1)_1$ for $(x,y) \in T_1 ,$ and $\eta$ be as in \eqref{proof of f^-1: -1}. 
Define
\begin{equation}\label{proof of K_f :6}
A_1 (x,y) = \frac{\tilde{\ell} (y)}{\alpha_t} \left(\eta^{-1} (u) -(\tilde{P}_1)_1 \right) + \tilde{x}_p \qquad \forall (x,y) \in T_1 .
\end{equation} 
By \eqref{proof of K_f :6} and \eqref{proof of K_f :2},
we have that
\begin{equation}\label{proof of K_f :7}
A= (A_1 , A_2 )
\end{equation}
is a diffeomorphism from $T_1$ onto $\tilde{T}_1 .$ 
We next give some estimates for $A .$
By \eqref{proof of K_f :1} we have that
\begin{equation}\label{proof of K_f :3}
\frac{\partial A_2 (x,y)}{\partial y} = \frac{\beta_t}{4 \delta_t \gamma_t} \approx \frac{t^{2s-1}}{\delta_t}  \qquad \forall (x,y) \in T_1.
\end{equation}
From \eqref{proof of f^-1: 0}, \eqref{proof of K_f :4-1} and \eqref{proof of K_f :1} it follows that
\begin{equation}\label{proof of K_f :8}
\frac{\partial A_1 (x,y)}{\partial x} = \frac{\tilde{\ell} (y)}{\alpha_t} (\eta^{-1})' (u) \frac{\partial u}{\partial x} \approx \frac{\tilde{\ell} (y)}{t} \qquad \forall (x,y) \in T_1.
\end{equation}
Moreover, by \eqref{proof of K_f :4-0} and \eqref{proof of K_f :4-1} we have that
\begin{equation}\label{proof of K_f :4}
\frac{\partial x_p}{\partial y} =-1,\ 
\frac{\partial \tilde{x}_p}{\partial y} =\frac{\beta_t -2\alpha_t}{\beta_t} \frac{\partial A_2}{\partial y} ,\ 
\frac{\partial \ell (y)}{\partial y} =2 \mbox{ and } 
\frac{\partial \tilde{\ell}(y)}{\partial y} = \frac{4 \alpha_t -2 \beta_t}{\beta_t}  \frac{\partial A_2}{\partial y}.
\end{equation}
It follows from \eqref{proof of K_f :4} that 
\begin{align}\label{proof of K_f :9}
\frac{\partial A_1}{\partial y}  =  & \frac{\partial \tilde{x}_p}{\partial y} +  \frac{\partial \tilde{\ell} (y)}{\alpha_t \partial y} \left(\eta^{-1} (u) - (\tilde{P}_1 )_1 \right) + \frac{\tilde{\ell}(y)}{\alpha_t} (\eta^{-1})' (u) \frac{\partial u}{\partial y} \notag\\
= & \frac{2\alpha_t -\beta_t}{\beta_t} \frac{\partial A_2}{\partial y} \left[ -1+\frac{2}{\alpha_t} (\eta^{-1} (u) -(\tilde{P}_1)_1)\right] \notag \\
& + \frac{\gamma_t \tilde{\ell}(y)}{\alpha_t \ell (y)} (\eta^{-1})' (u) \left[ 1- \frac{2}{\ell (y)} (x-x_p)\right].
\end{align}
Notice that $0 \le \eta^{-1} (u) -(\tilde{P}_1)_1 \le \alpha_t$ and $0 \le x-x_p \le \ell(y)$ for all $ (x,y) \in T_1 .$
Therefore \eqref{proof of K_f :9} together with \eqref{proof of K_f :1} and \eqref{proof of K_f :3} implies 
\begin{equation}\label{proof of K_f :10}
\left|\frac{\partial A_1 (x,y)}{\partial y}  \right| \lesssim \frac{2\alpha_t -\beta_t}{\beta_t} \frac{\partial A_2 (x,y)}{\partial y} \approx \frac{t}{\delta_t} \qquad \forall (x,y) \in T_1 .
\end{equation}
We conclude from \eqref{proof of K_f :3}, \eqref{proof of K_f :8} and \eqref{proof of K_f :10} that
\begin{equation}\label{proof of K_f :11-0}
|DA(x,y)| \lesssim \max \left \{\left|\frac{\partial A_1}{\partial x} \right|,\ \left|\frac{\partial A_1}{\partial y}\right|,\ \left|\frac{\partial A_2}{\partial x}\right|,\ \left|\frac{\partial A_2}{\partial y} \right| \right \} \lesssim \frac{t}{\delta_t}
\end{equation}
and  
\begin{equation}\label{proof of K_f :11-1}
J_A (x,y) = \frac{\partial A_1}{\partial x} \frac{\partial A_2}{\partial y} \approx \frac{t^{2s-2} \tilde{\ell}(y)}{\delta_t}
\end{equation}
for all $t\ll 1$ and all $(x,y) \in T_1 .$
Moreover by \eqref{proof of K_f :11-0}, \eqref{proof of K_f :11-1} and \eqref{proof of K_f :4-1} we have that
\begin{equation}\label{proof of K_f :11}
K_A (x,y)= \frac{|DA(x,y)|^2}{J_A (x,y)} \lesssim \frac{t^{4-2s}}{\delta_t \tilde{\ell}(y)} \lesssim \frac{t^{4(1-s)}}{\delta_t}
\end{equation}
holds for all $t\ll 1$ and all $(x,y) \in T_1 .$

We next define a diffeomorphism from $T_2$ onto $\tilde{T}_2 .$
Denote by $P_c$ and $\tilde{P}_c$ be the center of $R_t$ and $\tilde{R}_t  ,$ respectively.
Given $(x,y) \in T_2,$ we define
\begin{equation*}
B_1 (x,y) = \frac{2 \alpha_t -\beta_t}{4 \delta_t \gamma_t} \left(x- (P_c)_1 - \frac{\gamma_t}{2} \right) + (\tilde{P}_c)_1 + \frac{\alpha_t}{2} ,\ B_2 (x,y) = y \frac{a(x-(P_c)_1) +b}{c( x-(P_c)_1)+d},
\end{equation*}
where $a,\ b,\ c,\ d $ satisfy
\begin{equation}\label{proof of K_f :13}
a \gamma_t (\frac{1}{2} -\delta_t) +b = \frac{\beta_t}{4} ,\
a \frac{\gamma_t}{2} +b = \frac{\beta_t}{2},\
c \gamma_t (\frac{1}{2} -\delta_t) +d = \gamma_t (\frac{1}{2} -\delta_t) ,\
c \frac{\gamma_t}{2} +d = \frac{\gamma_t}{2}.
\end{equation}
Then 
\begin{equation}\label{proof of K_f :15}
B=(B_1 , B_2)
\end{equation}
is a diffeomorphism from $T_2$ onto $\tilde{T}_2.$ 
By \eqref{proof of K_f :1} we have that
\begin{equation}\label{proof of K_f :16}
\frac{\partial B_1 (x,y)}{\partial x} = \frac{2\alpha_t -\beta_t}{4 \delta_t \gamma_t} \approx \frac{t}{\delta_t} \qquad \forall (x,y) \in T_2.
\end{equation}
Moreover, from \eqref{proof of K_f :13} and \eqref{proof of K_f :1} we have that
\begin{equation}\label{proof of K_f :17}
\frac{\partial B_2 (x,y)}{\partial y} = \frac{a(x-(P_c)_1) +b}{c( x-(P_c)_1)+d} \approx \frac{\beta_t}{\gamma_t} \approx t^{2s-1} 
\end{equation}
and 
\begin{equation}\label{proof of K_f :18}
\left| \frac{\partial B_2 (x,y)}{\partial x} \right|= \frac{|y (ad-bc)|}{[c( x-(P_c)_1)+d]^2} \lesssim \frac{\gamma_t b}{\gamma^2 _t} \approx t^{2s-1}
\end{equation}
for all $(x,y) \in T_2.$
We then conclude from \eqref{proof of K_f :16}, \eqref{proof of K_f :17} and \eqref{proof of K_f :18} that
\begin{equation}\label{proof of K_f :19-0}
|DB (x,y)| \lesssim \max \left \{\left|\frac{\partial B_1}{\partial x} \right|,\ \left|\frac{\partial B_1}{\partial y}\right|,\ \left|\frac{\partial B_2}{\partial x}\right|,\ \left|\frac{\partial B_2}{\partial y} \right| \right \} \lesssim \frac{t}{\delta_t}
\end{equation}
and 
\begin{equation}\label{proof of K_f :19-1}
J_B (x,y)  = \frac{\partial B_1}{\partial x} \frac{\partial B_2}{\partial y} \approx \frac{t^{2s}}{\delta_t}.
\end{equation}
for all $t \ll 1$ and all $(x,y) \in T_2 .$
Moreover by \eqref{proof of K_f :19-0} and \eqref{proof of K_f :19-1} we have that 
\begin{equation}\label{proof of K_f :19}
K_B (x,y) = \frac{|DB(x,y)|^2}{J_B (x,y)} \lesssim \frac{t^{2(1-s)}}{\delta_t}
\end{equation}
for all $t \ll 1$ and all $(x,y) \in T_2 .$

We next construct a diffeomorphism $C : T_0 \rightarrow \tilde{T}_0 .$ 
By \eqref{proof of K_f :7} and \eqref{proof of K_f :15}
we have that $Q_1 Q_2$ is mapped onto $\tilde{Q}_1 \tilde{Q}_2$ by $A_1 (\cdot, \gamma_t(1/2 -\delta_t),$ and $Q_2 Q_3$ is mapped onto $\tilde{Q}_2 \tilde{Q}_3$ by $B_2 ((P_c)_1 + \gamma_t(1/2 -\delta_t) , \cdot).$
For a given $(x,y) \in T_0 , $ define
\begin{equation}\label{definition of C}
C(x,y)= \left(A_1 \big(x, \gamma_t(\frac{1}{2} -\delta_t) \big), B_2 \big( (P_c)_1 + \gamma_t(\frac{1}{2} -\delta_t) , y \big) \right).
\end{equation}
Then $C : T_0 \rightarrow \tilde{T}_0$ is diffeomorphic. 
By \eqref{proof of K_f :8} and \eqref{proof of K_f :17}, we have that 
\begin{equation*}
\frac{\partial }{\partial x} A_1 (x, \gamma_t(1/2 -\delta_t) \approx t^{2s -1},\ \frac{\partial }{y} B_2 ((P_c)_1 + \gamma_t(1/2 -\delta_t) , y) \approx t^{2s -1}
\end{equation*}
for all $(x,y) \in T_0 .$
Therefore 
\begin{equation}\label{proof of K_f :20}
|DC(x,y)| \lesssim t^{2s-1}\mbox{ and } K_C (x,y) \approx 1 
\end{equation}
for all $t \ll 1$ and all $(x,y) \in T_0 .$

Via \eqref{proof of K_f :7}, \eqref{proof of K_f :15} and \eqref{definition of C}, we redefine $f^{-1} _4: R_t \rightarrow \tilde{R}_t $ in \eqref{proof of f^-1: 10-0} as
\begin{equation}\label{proof of K_f :20-1}
f^{-1} _4 (x,y) =
\begin{cases}
A(x,y) &\forall (x,y) \in T_1 , \\
B(x,y) & \forall  (x,y) \in T_2 ,\\
\left(A_1 (x,-y) ,-A_2 (x,-y) \right), & \forall(x,y) \in T_3 ,\\
(2 (\tilde{P}_c)_1 -B_1 (2 (P_c)_1 -x,y), B_2 (2 (P_c)_1 -x,y))  &  \forall (x,y) \in T_4 ,\\
C(x,y)& \forall (x,y) \in T_0 .
\end{cases}
\end{equation} 
Like in Section \ref{existence}, by taking a fixed $j_0 \gg 1$ we then define $F_{2^{-j}} :Q_{2^{-j}} \rightarrow \tilde{Q}_{2^{-j}}$ for all $j \ge j_0$, $E_1 :\Omega_1 \rightarrow \tilde{\Omega}_1$, $E_2 :\Omega_2 \rightarrow \tilde{\Omega}_2$ and $E :\mbr^2 \rightarrow \mbr^2 .$ 
It is not difficult to see that the new-defined $E$ is a homeomorphism such that $E(z) =z^2$ for all $ z \in \overline{M_s}$ and satisfies \eqref{MFD :1} for $E$ on $\mathcal{L}^2$-a.e. $\mathbb{R}^2 .$ 
To show $E \in \mathcal{E}_s ,$ it is then enough to prove that $E \in W^{1,1} _{\loc} (\mbr^2 , \mbr^2)$ and $J_E \in L^1 _{\loc} (\mbr^2) .$
By \eqref{proof of f^-1: 6}, \eqref{proof of f^-1: 10}, \eqref{proof of K_f :11-0}, \eqref{proof of K_f :19-0} and \eqref{proof of K_f :20}, we have that
\begin{align}\label{proof of K_f :20-2}
DF_{2^{-j}} (z) = & Df^{-1} _3 (f^{-1} _4 \circ f_2 \circ f^{-1} _1 (z)) D f^{-1} _4 ( f_2 \circ f^{-1} _1 (z)) D (f_2 \circ f^{-1} _1) (z) \notag \\
\lesssim &
\begin{cases}
\frac{2^{-j}}{\delta_{2^{-j}}} &  \mathcal{L}^2 \mbox{-a.e. }z \in f_1 \circ f^{-1} _2 (\cup_{k=1} ^{4} T_k) , \\
2^{j(1-2s)}  & \mathcal{L}^2 \mbox{-a.e. }z \in f_1 \circ f^{-1} _2 (T_0) ,
\end{cases}
\end{align}
for all $j \ge j_0.$
Notice that
\begin{equation*}
\mathcal{L}^2 (T_0 ) = (\gamma_{2^{-j}} (1-2\delta_{2^{-j}}))^2 \approx 2^{-2j} ,\
\mathcal{L}^2 (T_k) = \delta_{2^{-j}} \gamma^2 _{2^{-j}} (1-\delta_{2^{-j}}) \approx \delta_{2^{-j}} 2^{-2j}
\end{equation*}
for all $k=1,2,3,4 $ and all $j \ge j_0 .$ It hence follows from \eqref{proof of f^-1: 4} that 
\begin{equation}\label{proof of K_f :22}
\mathcal{L}^2 (f_1 \circ f^{-1} _2 (T_0)) \approx 2^{-2j} ,\ 
\mathcal{L}^2 (f_1 \circ f^{-1} _2 (T_k)) \approx \delta_{2^{-j}} 2^{-2j} \quad \mbox{for all } k=1,2,3,4.
\end{equation}
By \eqref{proof of K_f :20-2} and \eqref{proof of K_f :22} we then have that
\begin{equation*}\label{proof of K_f :23-1}
\int_{Q_{2^{-j}}} |DF_{2^{-j}}| = \sum_{k=0} ^{4} \int_{f_1 \circ f^{-1} _2 (T_k)} |DF_{2^{-j}}| \lesssim 2^{-3j} + 2^{-j(2s+1)} \lesssim 2^{-3j} \qquad \forall j \ge j_0 .
\end{equation*}  
Therefore
\begin{equation}\label{proof of K_f :23-2}
\int_{\Omega_1} |DE_1| = \sum_{j=j_0} ^{\infty} \int_{Q_{2^{-j}}} |DF_{2^{-j}}| \lesssim \sum_{j=j_0} ^{\infty}  2^{-3j} <\infty .
\end{equation}
By \eqref{key tool: 1}, \eqref{proof of K_f :23-2} and the fact that $E(z)=z^2$ for all $ z \in M_s ,$ we have that $E \in W^{1,1} _{\loc} (\mathbb{R}^2 , \mbr^2) .$
Analogously to \eqref{proof of f^-1: 12-4}, we have that
\begin{equation}\label{proof of K_f :23-3}
\int_{\Omega_1} |J_{E_1}| \le \mathcal{L}^2(\tilde{\Omega}_1) <\infty .
\end{equation}
From \eqref{key tool: 1}, \eqref{proof of K_f :23-3} and the fact that $E(z)=z^2$ for all $z \in M_s ,$ we have that $J_E \in L^1 _{\loc} (\mathbb{R}^2) .$

We next show $K_{E} \in L^q _{\loc} (\mathbb{R}^2)$ for all $q < 1.$ By \eqref{proof of f^-1: 7}, \eqref{proof of f^-1: 10}, \eqref{proof of K_f :11}, \eqref{proof of K_f :19} and \eqref{proof of K_f :20}, we have that
\begin{equation}\label{proof of K_f :21}
K_{F_{2^{-j}}} (z)\lesssim
\begin{cases}
\frac{2^{4j(s-1)}}{\delta_{2^{-j}}} & \forall \ z \in f_1 \circ f^{-1} _2 (T_1 \cup T_3),\\
\frac{2^{2j(s-1)}}{\delta_{2^{-j}}} &\forall \ z \in f_1 \circ f^{-1} _2 (T_2 \cup T_4), \\
1 &\forall \ z \in f_1 \circ f^{-1} _2 (T_0).
\end{cases}
\end{equation}
for all $j \ge j_0 .$
For any $q \ge 0,$ via \eqref{proof of K_f :22} and \eqref{proof of K_f :21} we obtain that
\begin{equation*}\label{proof of K_f :21-1}
\int_{Q_{2^{-j}}} K^q _{F_{2^{-j}}} = 
\sum_{k=0} ^{4} \int_{f_1 \circ f^{-1} _2 (T_k)} K^q _{F_{2^{-j}}} 
\lesssim  \delta^{1-q} _{2^{-j}} 2^{j(4q (s-1) -2)} (1+2^{2qj(1-s)}) +2^{-2j} 
\end{equation*}
for all $j \ge j_0 .$
Therefore
\begin{align}\label{proof of K_f :21-2}
\int_{\Omega_1} K^q _{E} = & \sum_{j=j_0} ^{+\infty} \int_{Q_{2^{-j}}} K^q _{F_{2^{-j}}} \notag \\
\lesssim & \sum_{j=j_0} ^{+\infty} \exp((q-1)2^j) 2^{j(4q (s-1) -2)} (1+2^{j2q (1-s)}) +\sum_{j=j_0} ^{+\infty} 2^{-2j} <+\infty
\end{align}
for all $q \in (0,1)$ and each $s > 1.$
By \eqref{key tool: 1}, \eqref{proof of K_f :21-2} and the fact that $E$ is conformal on $M_s ,$ we conclude that $K_E \in L^q _{\loc} (\mbr^2)$ for all $ q \in (0,1) .$ 
\end{proof}

\subsection{\eqref{theorem Df+K_f :1}}

\begin{proof}[Proof of \eqref{theorem Df+K_f :1}]
Let $g :\mathbb{D} \rightarrow \Delta_s$ be conformal, where $\Delta_s$ is defined in \eqref{definition of M_s and Delta_s} with $s >1 .$
In order to prove \eqref{theorem Df+K_f :1}, via Lemma \ref{rmk inverse} ($4$) it is enough to construct $f \in \mathcal{F}_s (g)$ such that $f \in W^{1,p} _{\loc} (\mathbb{R}^2 ,\mbr^2)$ for some $p>1$ and $K_{f} \in L^q _{\loc} $ for all $q < \max\{1/(s-1), 3p/((2s-1)p+4-2s) \} .$ 

We consider the case $s \in (1,2]$ first.
Let $G_s$ be as in \eqref{rmk inverse: 0} and $\mathfrak{M}_s$ be as in \eqref{mathfrakM_s }. 
If $E \in \mathcal{E}_s$ satisfying that $E \in W^{1,p} _{\loc} (\mathbb{R}^2 , \mbr^2)$ for some $p>1$ and $K_{E} \in L^q _{\loc} $ for all $q < 1/(s-1) ,$
by Lemma \ref{composition} and the analogous arguments as for \eqref{rmk2 :9} and \eqref{rmk2 :10}, we can define $ f = E \circ G_s \circ \mathfrak{M}_s .$ 
We now let $E$ be as in \eqref{definition of E}. Then $E \in \mathcal{E}_s .$
By \eqref{proof of f^-1: 12-3}, \eqref{key tool: 1} and the fact that $E(z)=z^2$ for all $ z \in M_s,$ we obtain that $E \in W^{1,p} _{\loc} (\mathbb{R}^2 , \mbr^2)$ for all $p \ge 1 .$
From \eqref{proof of f^-1: 10}, \eqref{proof of f^-1: 12} and \eqref{proof of f^-1: 7}, it follows that  
\begin{equation*}
K_{F_{2^{-j}}} (z) = K_{f^{-1} _3} (f^{-1} _4 \circ f_2 \circ f^{-1} _1 (z)) K_{f^{-1} _4} (f_2 \circ f^{-1} _1 (z))  K_{f_2 \circ f^{-1} _1} (z) \approx 2^{(2s-2)j}
\end{equation*}
for all $j \ge j_0 $ and $\mathcal{L}^2$-a.e. $z \in Q_{2^{-j}} .$
Together with $\mathcal{L}^2(Q_{2^{-j}})\approx 2^{-2j},$ we then obtain
\begin{equation}\label{not know1}
\int_{\Omega_1} K^q _E = \sum_{j=j_0} ^{+\infty} \int_{Q_{2^{-j}}} K^q _{F_{2^{-j}}} \approx \sum_{j=j_0} ^{+\infty} 2^{-j2(1+q(1-s))} <\infty
\end{equation}
for all $q < 1/(s-1).$
By \eqref{not know1}, \eqref{key tool: 1} and the fact that $E$ is conformal on $M_s,$ we have that $K_E \in L^q _{\loc} (\mathbb{R}^2) $ for all $q <1/(s-1) .$

We turn to the case $s > 2.$ Let $M(p,s) = 3p/((2s-1)p+4-2s) $ with $p>1 .$
Analogously to the case $s \in (1,2],$ it is enough to construct $E \in \mathcal{E}_s$ such that $E \in W^{1,p} _{\loc} (\mathbb{R}^2 ,\mbr^2)$ and $K_E \in L^q _{\loc} (\mathbb{R}^2)$ for all $q \in (0, M(p,s)) .$ 
Redefining $\delta_t$ in \eqref{delta_t} as $\delta_t=t^{\frac{p+2}{p-1}} \log^{\frac{p}{p-1}} (t^{-1}).$ We follow the methods in Section \ref{proof of delta_1} to define a new $f^{-1} _4 .$
Set $j_0 \gg 1.$ There are then new $F_{2^{-j}} :Q_{2^{-j}} \rightarrow \tilde{Q}_{2^{-j}}$ for all $j \ge j_0$, $E_1 :\Omega_1 \rightarrow \tilde{\Omega}_1$, $E_2 : \Omega_2 \rightarrow \tilde{\Omega}_2$ and $E :\mbr^2 \rightarrow \mbr^2 .$ 
It is not difficult to see that the new $E$ is homeomorphic, satisfies \eqref{MFD :1} for $E$ on $\mathcal{L}^2$-a.e. $\mathbb{R}^2$ and $J_E \in L^{1} _{\loc} (\mathbb{R}^2) .$
To show that $E$ satisfies all requirements, it is enough to check that
$E \in W^{1,p} _{\loc}(\mathbb{R}^2 , \mbr^2)$ and 
$K_E \in L^{q} _{\loc}(\mathbb{R}^2)$ for all $q \in (0,M(p,s)) .$

From \eqref{proof of f^-1: 6}, \eqref{proof of f^-1: 10}, \eqref{proof of K_f :11-0}, \eqref{proof of K_f :19-0} and \eqref{proof of K_f :20} we have that
\begin{equation}\label{proof of K_f :27-1}
|DF_{2^{-j}} (z)| \lesssim
\begin{cases}
\frac{2^{-j}}{\delta_{2^{-j}}} &\forall z \in f_1 \circ f^{-1} _2 (\cup_{k=1} ^{4} T_k) ,\\
2^{j(1-2s)} & \forall  z \in f_1 \circ f^{-1} _2 (T_0) ,
\end{cases}
\end{equation}
for all $j \ge j_0 .$
It follows from \eqref{proof of K_f :27-1} and \eqref{proof of K_f :22} that
\begin{equation*}\label{proof of K_f :27}
\int_{Q_{2^{-j}}} |DF_{2^{-j}}|^p = \sum_{k=0} ^{4} \int_{f_1 \circ f^{-1} _2 (T_k)} |DF_{2^{-j}}|^p 
\lesssim  \delta^{1-p} _{2^{-j}} 2^{-j(2+p)} +2^{j(p(1-2s) -2)} .
\end{equation*}
Therefore
\begin{equation}\label{proof of K_f :28-1}
\int_{\Omega_1} |DE|^p =  \sum_{j=j_0} ^{+\infty} \int_{Q_{2^{-j}}} |DF_{2^{-j}}|^p  
\lesssim  \sum_{j=j_0} ^{+\infty} \frac{1}{j^p } + \sum_{j=j_0} ^{+\infty} 2^{-j(p(2s-1) +2)} <\infty . 
\end{equation}
By \eqref{proof of K_f :28-1}, \eqref{key tool: 1} and the fact that $E(z)=z^2$ for all $z \in M_s ,$ we conclude that $E \in W^{1,p} _{\loc}(\mathbb{R}^2 , \mbr^2) .$
By \eqref{proof of f^-1: 6}, \eqref{proof of f^-1: 7}, Lemma \ref{lemma A} and \eqref{proof of f^-1: 10}, we have 
\begin{align}\label{proof of K_f :30}
\int_{f_1 \circ f^{-1} _2 (T_1)} K^q _{F_{2^{-j}}} \approx & \int_{f_1 \circ f^{-1} _2 (T_1)} K^q _{f^{-1} _3} (f^{-1} _4 \circ f_2 \circ f^{-1} _1) K^q _{f^{-1} _4} (f_2 \circ f^{-1} _1) K^q _{f_2 \circ f^{-1} _1}  \big|J_{f_2 \circ f^{-1} _1} \big| \notag \\
\le & \int_{T_1}  K^q _{f^{-1} _3} (f^{-1} _{4} ) K^q _{f^{-1} _{4}}  \notag\\
\lesssim & \int_{T_1} K^q _{f^{-1} _{4}} 
\end{align}
for all $q \ge 0 $ and all $j \ge j_0 .$
Notice $\tilde{\ell}(\gamma_{2^{-j}}/2) =\alpha_{2^{-j}}$ and $\tilde{\ell}(\gamma_{2^{-j}}(\frac{1}{2} -\delta_{2^{-j}})) =\beta_{2^{-j}}/2$ for all $j \ge 1.$
By Fubini's theorem, \eqref{proof of K_f :11}, \eqref{proof of K_f :4-1} and \eqref{proof of K_f :1} we then have 
\begin{align}\label{proof of K_f :31}
\int_{T_1} K^q _{f^{-1} _{4}} \lesssim & \int_{\gamma_{2^{-j}}(\frac{1}{2} -\delta_{2^{-j}})} ^{\frac{\gamma_{2^{-j}}}{2}}\int_{x_p} ^{x_p +\ell (y)} \left(\frac{2^{j(2s-4)}}{\delta_{2^{-j}} \tilde{\ell} (y)} \right)^q  \, dx \, dy \notag\\
\approx & \frac{ 2^{jq(2s-4)} \gamma_{2^{-j}}} {\delta^q _{2^{-j}}} \int_{\gamma_{2^{-j}}(\frac{1}{2} -\delta_{2^{-j}})} ^{\frac{\gamma_{2^{-j}}}{2}} \frac{1}{\tilde{\ell}^q (y)}\, dy \notag\\
=& \frac{ 2^{jq(2s-4)} \gamma_{2^{-j}}}{(1-q)\delta^q _{2^{-j}}} \frac{2 \delta_{2^{-j}} \gamma_{2^{-j}}}{2\alpha_{2^{-j}}-\beta_{2^{-j}}} \left(\tilde{\ell}^{1-q} (\frac{\gamma_{2^{-j}}}{2}) - \tilde{\ell}^{1-q}(\gamma_{2^{-j}}(\frac{1}{2} -\delta_{2^{-j}})) \right) \notag\\
\lesssim & \frac{\delta^{1-q} _{2^{-j}} 2^{-2j[1+q(1-s)]}}{1-M(p,s)} 
\end{align}
for any fixed $q \in (0,M(p,s)).$ 
Combining \eqref{proof of K_f :30} with \eqref{proof of K_f :31} implies that
\begin{equation}\label{proof of K_f :32}
\int_{f_1 \circ f^{-1} _2 (T_1)} K^q _{F_{2^{-j}}} \lesssim \delta^{1-q} _{2^{-j}} 2^{-2j[1+q(1-s)]} \qquad \forall j \ge j_0.
\end{equation}
By symmetry of $f^{-1} _4$ between $T_1$ and $T_3,$ it follows from \eqref{proof of K_f :32} that 
\begin{equation}\label{proof of K_f :33}
\int_{f_1 \circ f^{-1} _2 (T_3)} K^q _{F_{2^{-j}}} = \int_{f_1 \circ f^{-1} _2 (T_1)} K^q _{F_{2^{-j}}} \lesssim \delta^{1-q} _{2^{-j}} 2^{-2j[1+q(1-s)]} 
\end{equation}
for all $j \ge j_0 .$
By \eqref{proof of K_f :21} and \eqref{proof of K_f :22}, we have that
\begin{equation}\label{proof of K_f :28}
\int_{f_1 \circ f^{-1} _2 (T_0)} K^q _{F_{2^{-j}}} \lesssim 2^{-2j}
\end{equation}
and 
\begin{equation}\label{proof of K_f :29}
\int_{f_1 \circ f^{-1} _2 (T_2 \cup T_4)} K^q _{F_{2^{-j}}} \lesssim \delta_{2^{-j}} 2^{-2j} \left(\frac{2^{2j(s-1)}}{\delta_{2^{-j}}} \right)^q  
= \delta^{1-q} _{2^{-j}} 2^{2j[q(s-1)-1]} 
\end{equation}
for all $j \ge j_0 .$
From \eqref{proof of K_f :32}, \eqref{proof of K_f :33}, \eqref{proof of K_f :28} and \eqref{proof of K_f :29}, we conclude that
\begin{align}\label{proof of K_f :38}
\int_{\Omega_1} K^q _{E} = & \sum_{j=j_0} ^{+\infty} \int_{Q_{2^{-j}}} K^q _{F_{2^{-j}}} = \sum_{j=j_0} ^{+\infty} \sum_{k=0} ^{4} \int_{f_1 \circ f^{-1} _2 (T_k)} K^q _{F_{2^{-j}}} \notag\\
\lesssim & \sum_{j=j_0} ^{+\infty} 2^{-2j} +  2^{-j \left( \frac{(p+2)(1-q)}{p-1} +2[1+q(1-s)]\right)} \log^{\frac{p(1-q)}{p-1}} \left(2^j\right).
\end{align}
Note that 
\begin{equation*}
\frac{(p+2)(1-q)}{p-1} +2[1+q(1-s)] > 0 \Leftrightarrow q < M(p,s).
\end{equation*}
It from \eqref{proof of K_f :38} follows that $\int_{\Omega_1} K^q _{E}  < \infty$ for all $q \in (0, M(p,s) ) .$
Together with \eqref{key tool: 1} and the fact that $E$ is conformal on $M_s ,$ we conclude that $K_E \in L^{q} _{\loc}(\mathbb{R}^2)$ for all $q \in (0,M(p,s)) .$ 
\end{proof}

\section{Proof of Theorem \ref{cardioid K_f}}

\begin{proof}
Let $\Delta$ be as in \eqref{Delta}.
The representation of $\partial \Delta$ in Cartesian coordinates is
\begin{equation*}
(x^2 +y^2)^2 -4x (x^2 +y^2) -4y^2 =0 .
\end{equation*}
Hence we can parametrize $\partial \Delta$ in a neighborhood of the origin as 
\begin{equation*}
\tilde{\Gamma}_0 =\{(x,y)\in \mathbb{R}^2 : x \in [-2^{-j_0},0] , y^2 =d(x)\},
\end{equation*}
where $j_0 \gg 1$ and $d(x) =\frac{- x^3 (4-x)}{2-x^2 +2x + \sqrt{1+2x}} .$
Since $d (x) \approx |x|^3 $ for all $ |x| \ll 1 ,$ 
there are $c_1 >0,\ c_2 >0$ such that 
\begin{equation*}
-c _1  x^3 \le d(x) \le -c _2 x^3 \qquad \forall x \in [-2^{-j_0},0] .
\end{equation*}
Denote 
\begin{equation*}
\tilde{\Gamma}_1 = \{(x,y)\in \mathbb{R}^2 :x \in [-2^{-j_0},0] , y^2 = -c _1 x^3\},
\end{equation*}
\begin{equation*}
\tilde{\Gamma}_2 = \{(x,y)\in \mathbb{R}^2 :x \in [-2^{-j_0},0] , y^2 = -c _2 x^3\},
\end{equation*}
\begin{equation*}
\tilde{\Gamma}_3 = \{(x,y)\in \mathbb{R}^2 : x= -2^{-j_0}, y^2 \in [c _1 (2^{-j_0})^3 , d(-2^{-j_0}) \} ,
\end{equation*}
\begin{equation*}
\tilde{\Gamma}_4 = \{(x,y)\in \mathbb{R}^2 : x= -2^{-j_0}, y^2 \in [d(-2^{-j_0}) , c _2 (2^{-j_0})^3] \} .
\end{equation*}
Let $\tilde{\Omega}_u$ and $\tilde{\Omega}_d$ be the domains bounded by $ \tilde{\Gamma}_0 \cup \tilde{\Gamma}_2 \cup \tilde{\Gamma}_4$ and $\tilde{\Gamma}_0 \cup \tilde{\Gamma}_1 \cup \tilde{\Gamma}_3 ,$ respectively.
Denote by $\Omega_u , \Omega_d$ and $\Gamma_k$ for $k=0,...,4$ the images of $\tilde{\Omega}_u , \tilde{\Omega}_d$ and $\tilde{\Gamma}_k$ under the branch of complex-valued function $z^{1/2}$ with $1^{1/2} =1 ,$ respectively.

We first prove the existence of an extension, see FIGURE \ref{extension}.
\begin{figure}[htbp]
\centering \includegraphics[bb=142 509 584 678, clip=true, scale=0.8]{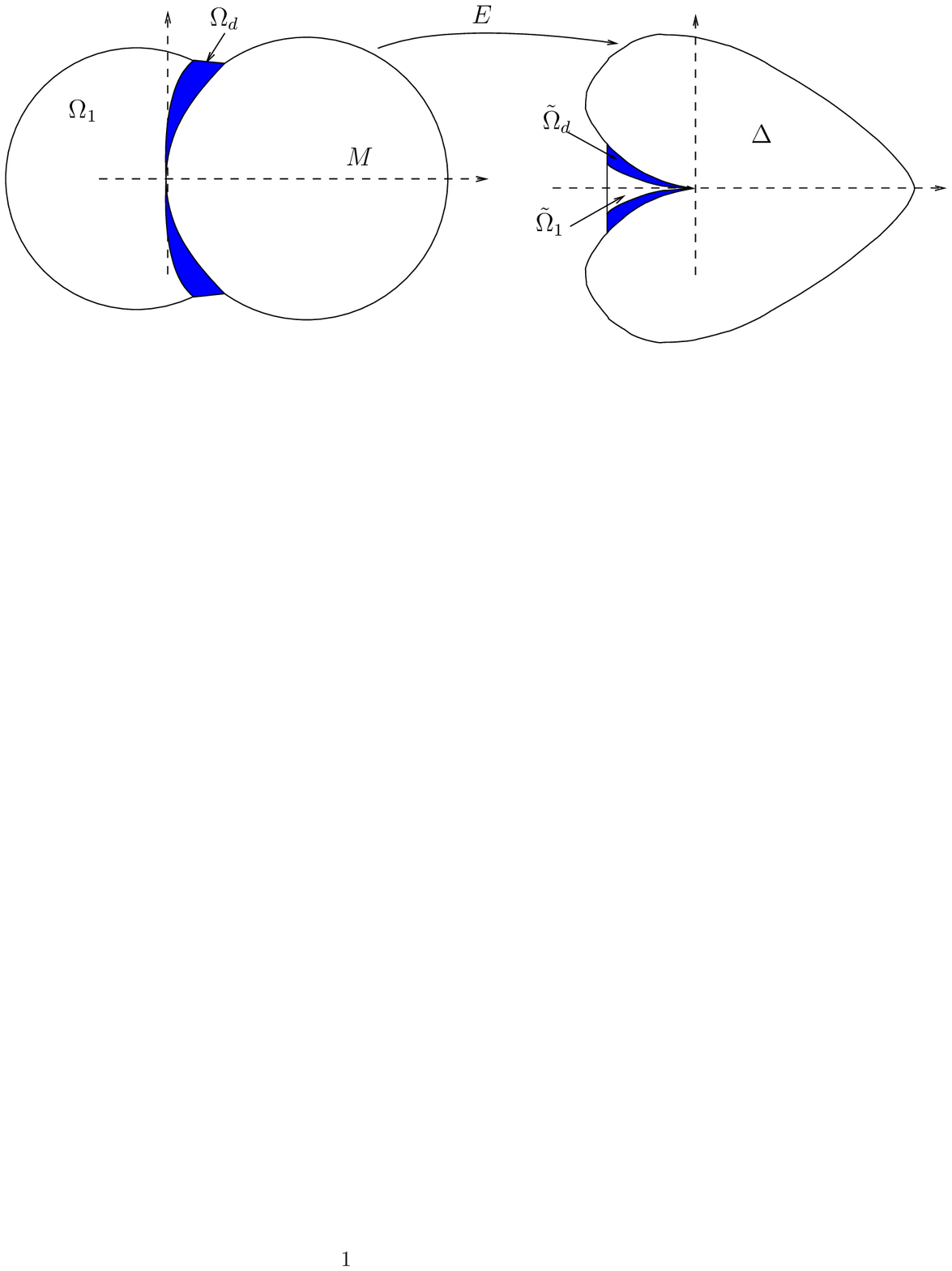}

\caption{The existence of an extension}
\label{extension}
\end{figure}
Let $r= (2^{-2j_0} +c _1 2^{-3 j_0})^{1/4} .$ Denote
\begin{equation*}
M =\{(x+1,y)\in \mathbb{R}^2 : (x,y) \in \mathbb{D}\} ,
\end{equation*}
\begin{equation*}
\Omega_1 = \overline{B(0, r)} \setminus (M \cup \Omega_d)
,\
\Omega_2 = \mathbb{R}^2 \setminus (\Omega_1 \cup \Omega_d \cup M),
\end{equation*}
\begin{equation*}
\tilde{\Omega}_1 = \{(x,y)\in \mathbb{R}^2 : x \in [-2^{-j_0} ,0], y^2 \le c _1 |x|^3 \} \mbox{ and } \tilde{\Omega}_2 = \mathbb{R}^2 \setminus (\tilde{\Omega}_1 \cup \tilde{\Omega}_d \cup \Delta) .
\end{equation*}
Analogously to the arguments in Section \ref{existence}, we define $E_1 : \Omega_1 \rightarrow \tilde{\Omega}_1$ and $E_2: \Omega_2 \rightarrow \tilde{\Omega}_2 .$ 
Here $\eta(x) = \sqrt{x} (1+c _1 x)^{1/4}$ and $s=3/2 .$ 
Define
\begin{equation}\label{proof of cardioid :1}
E(x,y)=
\begin{cases}
E_1 (x,y) &\forall \ (x,y) \in \Omega_1 ,\\
E_2 (x,y) &\forall \ (x,y) \in \Omega_2 ,\\
(x^2 -y^2 ,2xy) &\forall \ (x,y) \in M \cup \Omega_d ,
\end{cases}
\end{equation}
and $f_0 (x,y) =E(x+1 ,y).$
By the analogous arguments as in Section \ref{existence}, we have that $f_0 \in \mathcal{F}.$

We next prove \eqref{cardioid K_f :1}.
Suppose $f \in \mathcal{F} .$ Then $\hat{f} (u,v)=f(u-1,v)$ is a homeomorphism of finite distortion on $\mathbb{R}^2 $ and $\hat{f} ( M \setminus \Omega_u) = \Delta \setminus \tilde{\Omega}_u.$
By Remark \ref{lem 3.1}, we have that if $K_{\hat{f}} \in L^q _{\loc} (\mathbb{R}^2 ) $ then $q < 2 .$ 
Therefore if $K_f \in L^q _{\loc} (\mathbb{R}^2 ) $ then $q < 2 .$ 
In order to prove \eqref{cardioid K_f :1}, it then suffices to 
construct a mapping $f_0 \in \mathcal{F} $ such that $K_{f_0} \in L^q _{\loc} (\mathbb{R}^2)$ for all $q <2 .$
Let $E$ be as in \eqref{proof of cardioid :1} and $f_0 (x,y)=E(x+1 ,y) .$ 
Then $f_0 \in \mathcal{F} .$  
The same arguments as for the case $s \in (1,2)$ in Section \ref{proof of delta_1} show that $K_{E} \in L^{q} _{\loc} (\mathbb{R}^2) $ for all $q <2 .$ Therefore $K_{f_0} \in L^{q} _{\loc} (\mathbb{R}^2) $ for all $q <2 .$

The strategies to prove \eqref{cardioid K_f :0}, \eqref{cardioid K_f :2}, \eqref{cardioid K_f :3} and \eqref{cardioid K_f :4} are same as the one to prove \eqref{cardioid K_f :1}. We leave the details to the interested reader. 
\end{proof}

\section*{Acknowledgment}
The author has been supported by China Scholarship Council (project No. $20170634$ $0060$). 
This paper is a part of the author's doctoral thesis. 
The author thanks his advisor Professor Pekka Koskela for posing this question and for valuable discussions. The author thanks Aleksis Koski and Zheng Zhu for comments on the earlier draft.

\bigskip
\bibliographystyle{amsplain}

\noindent Haiqing Xu

\noindent 
Department of Mathematics and Statistics, University of Jyv\"askyl\"a, PO BOX 35, FI-40014 Jyv\"askyl\"a, Finland

\noindent 
School of Mathematical Sciences, University of Science and Technology of China, Hefei 230026, P. R. China

\noindent{\it E-mail address}:  \texttt{hqxu@mail.ustc.edu.cn}
\end{document}